\documentclass[12pt]{article}
\usepackage[all]{xy}
\usepackage{amsmath}
\usepackage{amsfonts}
\usepackage{amssymb}
\usepackage{theorem}

\title{Groups which are metrically LEF}
\author{Aleksander Ivanov and Ireneusz Sobstyl
}

\newtheorem{theorem}{Theorem}[section]
\newtheorem{proposition}[theorem]{Proposition}
\newtheorem{corollary}[theorem]{Corollary}
\newtheorem{lemma}[theorem]{Lemma}
\newtheorem{definition}[theorem]{Definition}

\newtheorem{example}[theorem]{Example}
\newtheorem{remark}[theorem]{Remark}
\newenvironment{proof}{\addvspace{8pt plus 2pt minus 2pt}\noindent\emph{Proof. }}
  { \begin{flushright}$\blacksquare$\par\addvspace{8pt plus 2pt minus
2pt}\end{flushright}}

\begin{document} 
\topmargin = 12pt
\textheight = 630pt 
\footskip = 39pt 

\maketitle

\begin{quote}
{\bf Abstract} 
We consider metric versions of the notions of local embeddability and LEF. 
We pay special attention to normally finitely generated groups with word metrics.  
\end{quote}

\section{Introduction}

Let $G$ be a group and let $\Lambda$ be a closed convex subset of $[0,\infty)$ so that $0\in \Lambda$. 
\begin{definition} \label{lf}
{\em A function $\ell:G\rightarrow \Lambda$ is called a} pseudo-length function (or pseudo-norm) {\em if for all $g$ and $h \in G$ \\ 
(i) $\ell(1) = 0 $; \\ 
(ii) $\ell(g)=\ell(g^{-1})$; \\
(iii) $\ell(gh)\le \ell(g) + \ell(h)$. \\ 
A} length function (or norm) {\em is a pseudo-length function which satisfies:} 
\begin{quote} 
{\em (i') for all $g\in G$ we have $\ell(g) = 0 $ if and only if $g=1$.}   
\end{quote} 
\end{definition} 
If $\ell$ is a $\Lambda$-(pseudo)-norm on $G$ then we say that $(G,\ell )$ is a {\em (pseudo) normed group}. 
The pseudo-norm $\ell$ is {\em invariant} if  $\ell(h^{-1} gh) = \ell(g)$ for all $g,h\in G$. 
In this case it defines an invariant pseudo-metric by $d_{\ell}(g,h)=\ell(gh^{-1})$. 
Thus 
\[ 
\forall x,y,z (d_{\ell}(z \cdot x, z\cdot y) = d_\ell (x,y) = d_\ell (x\cdot z ,y\cdot z)). 
\] 
It becomes a metric if $\ell$ is a length function. 
On the other hand if $d$ is an invariant (pseudo)-metric on $G$ then the function $d(x,1)$ is an invariant (pseudo )-norm. 
In order to simplify notation it is convenient to work with (pseudo) norms instead of (pseudo) metrics. 
However there are situations when the word "metric" looks more appropriate than "norm".   
In particular we prefer "metric LEF" than "normed LEF".

In order to introduce the main topic of the paper we remind the reader the following definition. 

\begin{definition} 
{\em (\cite{GV})
A group $G$ is called} LEF 
{\em if for every 
finite subset $D\subseteq G$ there is 
a finite group $C$ and an injective map 
$\phi : D\rightarrow C$ such that 
any triple $h,g,hg\in D$ satisfies 
$\phi (hg )=\phi (h) \phi (g)$. }
\end{definition} 

Let us generalize this definition to pseudo-normed groups. 
This will give the main object of this paper. 

\begin{definition}  \label{mLEF} 
{\em 
Let $\mathcal{C}$ be a class of pseudo-normed groups (usually with invariant pseudo-norms). 
A group $(G,\ell_G )$ with an (invariant) pseudo-norm is called } 
locally embeddable into $\mathcal{C}$ 
{\em (shortly metrically LE$\mathcal{C}$) if for every finite subset $D\subseteq G$ and every finite $Q \subset \Lambda \cap \mathbb{Q}$ with $0\in Q$ there is a pseudo-normed group $(C,\ell_C )\in \mathcal{C}$ and an injective map $\phi : D\rightarrow C$ 
such that 
\begin{itemize} 
\item any triple $h,g,hg\in D$ satisfies 
$\phi (hg )=\phi (h) \phi (g)$, 
\item for any $g \in D$, any number $q\in Q$ and any symbol $\square \in \{ <\, , \, > \, , = \}$ we have   
\[ 
\ell_G (g) \square q \Leftrightarrow \ell_C (\phi (g)) \square q.  
\]
\end{itemize}  
} 
\end{definition} 
We will call the invariant pseudo-normed group $(G,\ell_G )$ a {\em metrically LEF group} if it is metrically LE$\mathcal{C}$ where $\mathcal{C}$ is the class of all finite invariant pseudo-normed groups. 
We will see in Section 2.3 that when $\ell_G$ is a norm this definition does not change under the additional demand that $\mathcal{C}$ consist of {\em all finite invariant normed groups}. 
This is the basic case in our paper.

\subsection{Word norms} 

There are metric/normed groups for which the question of metric LEF looks very interesting. 
For example it would be interesting to know how it fits to bounded normal generation (with norms as in \cite{DT}), in particular in the case of Chevalley groups, see  \cite{Ta}. 
Our original motivation concerns metrically (weakly) sofic groups, see \cite{Iv}, \cite{NST}.  
We also mention locally compact groups with property PL, see \cite{dC}.

Finitely generated groups with word norms are especially visible in (geometric) group theory, see \cite{BGKM}, \cite{BIP},  \cite{Karl}, \cite{LS}, \cite{Sha} and \cite{Trost}.  
For example when $G$ is generated by finitely many conjugacy classes then finiteness of its diameter for the appropriate word metric  implies boundedness of its diameter with respect to any invariant metric, \cite{BGKM}.  
Thus metric LEF in the case of word norms is natural for investigations. 
This leads to an interesting subclass of {\em abstract} finitely generated LEF groups: those ones which are metrically LEF with respect to a word norm associated with some tuple of generators. 
An unexpected fact: free groups are not there (see Section 3)! 
On the other hand there are metrically LEF groups which are not residually finite. 
In Section 4 we prove this for so called $G_2$, see p. 49 in \cite{CSC}. 
The proof takes one third part of the paper.   
In fact we obtain a theorem which works for a family of wreath products. 
The proof requires a deeper insight into this construction. 
New notions (for example $[k,g]$-commutators, where $k\in \mathbb{Z}$) and new tricks are invented.  
Some properties of finite simple groups (for example presentation as a product of two conjugacy classes) play a central role in the proof. 
Results of Sections 3 and 4 support  the following problem: 
\begin{itemize} 
\item {\em Describe metrially LEF groups for word norms in major group-theoretic classes. } 
\end{itemize} 
Although the material presented above looks mathematically attractive, a principal question remains, where the place of metrically LEF groups is in the world of group theory. 
We concern it in Secction 5. 
A more detailed description is as follows.  

\subsection{General theory} 
A systematic exposition of the theory of LEF groups is presented in Section 7 of the book \cite{CSC}. 
It is based on the language of marked groups and actively uses ultraproducts. 
Logical issues are not surprising in this context because LEF is a kind of the finite model property. 
The aim of Section 5 of our paper is to realize an exposition of ``generalities" of metrically LEF groups in the way parallel to that in \cite{CSC}. 
It turns out that the approach of \cite{CSC} should be substantially modified. 
The Grigorchuk space used there should be ``fused"  with the space of pseudo-length functions.   
More exactly, we view the set of invariant pseudo length functions $G \rightarrow \Lambda$ (for a countable group $G$) as a $G_{\delta}$-subspace of the compact space  
$ \{ < , >, = \}^{G\times (\mathbb{Q} \cap \Lambda )}$ (under the product topology). 
The set $\mathcal{N}(G)$ of all normal subgroups of $G$ is viewed as a closed subspace of the space $2^G$. 
To carry out a metric version of the approach of Section 7.1 of \cite{CSC} we will consider some closed subspace of the product  
$\mathcal{N}(G)\times \{ < , >, = \}^{G\times (\mathbb{Q} \cap \Lambda )}$.   

In Section 5.4 we modify this approach replacing length functions by families of binary relations. 
Then metric groups become firts-order structures. 
We will see some connections between logic properties of these structures and local embeddability.

\subsection{The structure of the paper} 
Section 2 collects basic preliminaries. 
We mention Section 2.3 which presents few useful tricks. 
It is proved in Section 3 that the two-generated free group with the standard word norm is not metrically LEF. 
In Section 4 we investigate metric LEF for wreath products $P  \wr  \mathbb{Z}$ where $P$ is a finite simple group. 
In Sections 5.1 - 5.3 we study metric local embeddability in the space of normed groups. 
In Section 5.4 we consider connections with the finite model property and in Section 5.5 we give a characterization of metric local embeddability in terms of direct limits. 

\section{Length functions on a group}

\subsection{Standard properties of norms}

The following lemma is obvious. 

\begin{lemma} \label{corr1}  
Let $H<G$. 
For any (invariant, pseudo) norm $\ell$ 
(metric $d$) on $G$ the restriction of $\ell$ (resp. $d$) to $H$ is an (invariant, pseudo) norm (resp. metric). 
\end{lemma} 

Let $(G, \ell)$ be an (invariant, pseudo) normed group, $g\in G$  and $r \in \Lambda$. 
The $r$-{\em ball of} $g$ is defined as follows: 
\[ 
B_r (g) = \{ h \in G \, | \, d_{\ell}(h,g) \le r \}. 
\] 
One can also consider the $(<r)$-{\em ball of} $g$: 
\[ 
B_{<r} (g) = \{ h \in G \, | \, d_{\ell}(h,g) < r \}. 
\] 
Then the family 
$\{ B_{<r}(g) \, | \, r\in \Lambda \cap \mathbb{Q} \}$ is a subbase of {\em the topology} defined by $\ell$. 
Thus every normed group is also a topological group.

\begin{definition} 
{\em Let $\ell$ be a pseudo-norm on $G$. 
We say that the subgroup 
$\mathsf{ker} (\ell )= \{ g\in G \, | \, \ell (g)=0 \}$ is the } kernel of $\ell$. 
\end{definition}

It is easy to see that when $\ell$ is invariant, $\mathsf{ker} (\ell )$ is a normal subgroup and for any $g_0\in G$, 
\[ 
g_0 \mathsf{ker} (\ell )= \{ g \, | \, d_{\ell} (g_0 , g) = 0\}. 
\]  
The following statements are taken from 
the paper of A.Stolz and A.Thom \cite{ST} (Lemmas 2.1 and 2.2).  

\begin{lemma} \label{corr2} 
(i) If $G$ is a group with an (invariant)  pseudo-norm $\ell$ and $H$ is a normal 
subgroup of $G$, then  
\[ 
\ell_{G/H} (gH )= \mathsf{inf} \{ \ell(gh): h\in H\} 
\] 
defines an (invariant) pseudo-norm on $G/H$. 
If $G$ is finite and $\ell$ is a norm, then $\ell_{G/H}$ is a norm too. 

(ii) If $G$ is a group and $H$ is a normal subgroup of $G$ so that $G/H$ has an (invariant) pseudo-norm $\ell$, then  
\[ 
\ell^{G} (g )= \ell(gH) 
\] 
is an (invariant) pseudo-norm on $G$ such that $H$ is a subgroup of its kernel. 
\end{lemma}

\subsection{Metric homomorphisms and almost-homomorphisms}

We now consider metric homomorphisms.

\begin{definition} \label{h} 
{\em Given two pseudo normed groups $(G_1 ,\ell_1 )$ and $(G_2 ,\ell_2 )$ we say that a map $\phi : G_1  \rightarrow G_2$ is a} (pseudo) metric homomorphism {\em if it is a group homomorphism and for any $h\in G_1$ we have 
$\ell_2 (\phi (h)) \le \ell_1 (h)$. 

If for any $h\in G_1$ the equality 
$\ell_2 (\phi (h)) = \ell_1 (h)$ 
holds, then we say that the homomorphism $\phi$ is} isometric. 
\end{definition}

Note that in the case when (pseudo-) norms $\ell_1$ and $\ell_2$ correspond to invariant pseudo-metrics $d_1$ and $d_2$ then the condition $(\forall h)(\ell_2 (\phi (h)) \le \ell_1 (h))$ 
(resp. $\ell_2 (\phi (h)) = \ell_1 (h)$)
is equivalent to 
$(\forall g,h)(d_2 (\phi (g), \phi (h)) \le d_1 (g,h))$ (resp. $d_2 (\phi (g), \phi (h)) = d_1 (g,h)$). 

It is also worth noting that an isometric homomorphism is not necessary injective. 
For example if $\ell$ is an invariant pseudo-norm then the natural homomorphism $G\rightarrow G/\mathsf{ker} (\ell)$ is isometric with respect to $\ell$ and $\ell_{G/\mathsf{ker} (\ell)}$.  

We now introduce a metric version of almost homomorphisms from \cite{CSC}, p. 240. 

\begin{definition}  \label{mLEFm} 
{\em Given two groups $G_1$ and $G_2$ with (invariant) pseudo-norms $\ell_1$ and $\ell_2$ respectively, let $D$ be a finite subset of $G_1$ and $Q$ be a finite subset of $\mathbb{Q} \cap \Lambda$ such that $0\in Q$. 
A map $\varphi :G_1 \rightarrow G_2$ is called a} $D$-$Q$-almost-homomorphism of pseudo-normed groups $(G_1 ,\ell_1 )$ and $(G_2 , \ell_2 )$ {\em if $\varphi$ is injective on $D$ and }
\begin{itemize} 
\item any triple $h,g,hg\in D$ satisfies 
$\varphi (hg )=\varphi (h) \varphi (g)$, 
\item for any $g \in D$, any number $q\in Q$ and any symbol $\square \in \{ <\, , \, > \, , = \}$ we have 
\[ 
\ell_1 (g) \square q \Leftrightarrow \ell_2 (\varphi (g)) \square q.  
\]
\end{itemize}  
\end{definition} 

This notion will be convenient in our arguments. 
The following easy lemma explains this. 

\begin{lemma} \label{alm_h}
Let $\mathcal{C}$ be a class of (invariant) pseudo normed groups. 
A group $(G,\ell_G )$ with an (invariant) pseudo-norm is metrically LE$\mathcal{C}$ 
if for every finite subset $D\subseteq G$ and every finite $Q \subset \Lambda \cap \mathbb{Q}$ with $0\in Q$ there is a pseudo normed group $(C,\ell_C )\in \mathcal{C}$ and a $K$-$Q$-almost-homomorphism $\varphi : (G,\ell_G )\rightarrow (C,\ell_C )$.  
\end{lemma} 
We emphasize that in general almost-homomorphisms are not metric homomorphisms. 
In the following definition we amalgamate these notions. 
It gives interesting subclasses of metrically LE$\mathcal{C}$ groups. 

\begin{definition} \label{mres}  
{\em Let $\mathcal{C}$ be a class of (invariant) pseudo-normed groups. 
A group $(G,\ell_G )$ with an (invariant) pseudo-norm is called} (metrically) fully residually $\mathcal{C}$ {\em if for every finite subset $D\subseteq G$ and every finite $Q \subset \Lambda \cap \mathbb{Q}$ with $0\in Q$ there is $(C,\ell_C )\in \mathcal{C}$ and a homomorphism $\varphi : G \rightarrow C$ 
(resp. metric homomorphism 
$\varphi: (G,\ell_G) \rightarrow (C,\ell_C )$) which is a $D$-$Q$-almost-homomorphism.} 
\end{definition} 
Note that to be metrically fully residually $\mathcal{C}$ implies the property to be fully residually $\mathcal{C}$. 
These details will be visible in Sections 5.2 and 5.3.  

\subsection{Norms for metric LEF} 

Let us denote by $\mathcal{C}_{(I)PMG}$ the class of {\em (invariant) pseudo-normed groups} and by $\mathcal{C}_{(I)MG}$ the class of {\em (invariant) normed groups}. 
Let $\mathcal{F}$ be the class of all finite pseudo-normed groups. 
According to Introduction {\em metric LEF} for an (invariant) pseudo-normed group $(G, \ell )$ means that $(G, \ell )$ is metrically LE$(\mathcal{C}_{(I)PMG}\cap \mathcal{F})$. 

In this section we pay special attention to norms which make a LEF group to be metrically LEF. 
Excluding the obvious examples (for example the $\{ 0,1\}$-norm) we concentrate on constructions which achieve some additional properties, in particular examples which illustrate definitions of the previous subsection.  
Using norm transformation we also collect some useful reductions (Propositions \ref{PSvsM} and \ref{RvsN}). 

It is obvious that every locally finite group with an invariant length function is metrically LE$(\mathcal{C}_{(I)MG}\cap \mathcal{F})$, i.e. it is metrically LEF. 
On the other hand it is easy to find such a group which is not fully residually  $\mathcal{C}_{(I)PMG}\cap \mathcal{F}$.  
The following statement gives examples of infinite finitely generated fully residually $\mathcal{C}_{(I)MG}\cap \mathcal{F}$ groups. 
Note that in this statement $G$ is given together with a natural topology on it and the norm $\ell$ corresponds to this topology. 

\begin{proposition} \label{refi}
Let $G$ be a residually finite group. 
Then there is an invariant norm $\ell$ such that $\ell$ defines the profinite topology on $G$ and $(G, \ell )$ is metrically fully residually $\mathcal{C}_{(I)MG}\cap \mathcal{F}$. 
\end{proposition} 

\begin{proof} 
Let $N_1 > N_2 > \ldots$ be a descending sequence of normal subgroups of finite index with $\bigcap N_i = \{ 1 \}$. 
It is a basis of neighborhoods of the identity with respect to the profinite topology. 

Let $p$ be any prime number. 
Define 
\[  
\ell_p (g) = \mathsf{max} \{ \frac{1}{p^s} \, | \, g\not\in N_{s}\}  .  
\] 
It is easy to see that $\ell_p$ is an invariant norm which defines the profinite topology. 
In order to verify that $(G, \ell_p )$ is metrically fully residually finite take any finite $K\subset G$. 
Let $N_s$ be chosen so that $K\cap N_s \subseteq \{ 1 \}$. 
By Lemma \ref{corr2}(i) we have 
$\ell (g) = \ell_{G/N_s}(gN_s )$ for any $g\in K$.
\end{proof} 

According to Introduction {\em metric LEF} for an (invariant) pseudo-normed group $(G, \ell )$ such that $\ell$ is not a norm, means that $(G, \ell )$ is metrically LE$(\mathcal{C}_{(I)PMG}\cap \mathcal{F})$. 
In the case when $\ell$ is a norm 
it is reasonable to demand additionally that $(G, \ell )$ is metrically LE$(\mathcal{C}_{(I)MG}\cap \mathcal{F})$. 
Proposition \ref{PSvsM} shows that these two possibilities are equivalent for normed groups.

\begin{proposition} \label{PSvsM}
If $(G, \ell )$ is a normed group, then it is metrically LE$(\mathcal{C}_{(I)PMG}\cap \mathcal{F})$ if and only if it is metrically LE$(\mathcal{C}_{(I)MG}\cap \mathcal{F})$.   
\end{proposition} 

\begin{proof} 
Assume that $\ell$ is an (invariant)  pseudo-norm on $G$ and $\varepsilon \in \mathbb{R}^{+}$. 
When 
\[ 
\varepsilon < \mathsf{min} (\ell (G) \setminus \{ 0 \})  
\]
define $\ell^{+\varepsilon} (x)$ to be $\ell (x)$  for all 
$x\in G \setminus \{g \, | \, \ell (g) = 0, \, g\not=1 \}$ but 
$\ell^{+\varepsilon} (x) = \varepsilon$ when $\ell (x) = 0$ and $x \not= 1$.   
It is an (invariant) $\Lambda$-norm. 

In order to verify the statement of the proposition take a finite $K\subset G$, a finite $Q\subset \Lambda \cap \mathbb{Q}$
and a $K$-$Q$-almost-homomorphism $\varphi$ from $(G,\ell )$ to some finite $(G_0 ,\ell_0 )$ with (invariant) pseudo-norm. 
Choose $\varepsilon \in \mathbb{Q}$ 
which is smaller than 
$\mathsf{min} (\ell_0 (G_0)\setminus \{ 0 \})$. 
It is easy to see that $\varphi$ is a $K$-$Q$-almost-homomorphism from $(G, \ell )$ to $(G_0 ,\ell^{+\varepsilon}_0 )$.  
\end{proof}

Some difficulty also arises when the set of values of $\ell$ is a proper subset of $\Lambda$, for example $\Lambda = \mathbb{R}^+$ and $\mathsf{Rng} (\ell) = \mathbb{N}$.  
{\em Does metrical LEF of $(G, \ell)$ imply that $(G, \ell )$ is metrically LEF with respect to the subclass of $\mathcal{C}_{(I)MG}$ consisting of $\mathsf{Rng}(\ell )$-normed groups?} 
The following proposition concerns one of the simplest cases. 

\begin{proposition} \label{RvsN} 
Assume that 
$\Lambda_0 = \Lambda \cap\mathbb{N}$.  
We also assume that if $\mathsf{sup} (\Lambda)$ exists then it belongs to $\Lambda_0$. 

Let $\ell$ be an (invariant) pseudo-norm on a group $G$ and 
$\Lambda_0 = \mathsf{Rng} (\ell )$.  
Assume that $(G,\ell )$ is metrically LE$(\mathcal{C}_{(I)PMG}\cap \mathcal{F})$. 

Then $(G,\ell )$ is metrically LEF with respect to the subclass of $\mathcal{C}_{(I)PMG}$ consisting of groups with pseudo-norms having the set of values in $\Lambda_0$. 
\end{proposition}  
 
\begin{proof}
Take a finite $K\subset G$, a finite $Q\subset \Lambda \cap \mathbb{Q}$ with $0\in Q$ and a $K$-$Q$-almost-homomorphism $\varphi$ from $(G,\ell )$ to some finite $(G_0 ,\ell_0 )$ with (invariant) pseudo-norm. 
It can happen that 
$\mathsf{Rng} (\ell_0 ) \not\subseteq \Lambda_0$. 
In order to correct this let us define a function $\ell'$ on $G_0$ as follows. 
If $g\in G_0$ and $\ell_0 (g) \in \Lambda_0$ put $\ell'(g) = \ell_0 (g)$. 

When $\ell_0 (g) \not\in \Lambda_0$  
let $n$ be the maximal natural number which is less than $\ell_0 (g)$. 
Define $\ell'(g') =n+1$ for all $g'\in G_0$ with $n < \ell_0 (g')\le n+1$.  
It is clear that this definition preserves (i) (resp. (i')), (ii) of Definition \ref{lf} and invariantness. 

In order to see property (iii) of Definition \ref{lf} let us apply induction. 
Let $m$ be the first number from $\Lambda_0$ such that $\ell_0$ takes values from $[0,m] \setminus \mathbb{N}$. 
Thus these reals belong to the open interval $(m-1,m)$. 
Since all values of $\ell_0$ which do not 
accede $m-1$, belong to $\mathbb{N}$, one easily sees that the (triangle) inequality (iii) holds for $\ell'$ and for all $g$, $h$ with $\ell_0 (gh) \le m$. 

If $g$ and $h$ satisfy 
$\ell_0 (gh) \in (m,m+1]$ then $\ell' (gh) = m+1$. 
If $\ell_0 (g) \not=\ell'(g)$ then $\ell'(g) = [\ell_0 (g)]+1$ (and so is true for $h$).  
In particular $\ell'(gh) \le \ell'(g) + \ell'(h)$.  
The rest of this induction is clear. 
\end{proof}  
 
This proposition will be applied in Section 3. 
It seems to us that the method used in the proof can work in many other situations. 
In fact we do not know any example where the natural version of Proposition \ref{RvsN} does not hold. 
For example, it is easy to see that the statement of this proposition holds in the case when $\Lambda_0$ is dense in $\Lambda$.

\section{Metrically LEF groups with respect to word norms} 

The following definition is taken from \cite{BGKM}. 
Let $G$ be a group generated by a symmetric set $S \subset G$. 
The {\em word norm} of an element $g \in G$ associated with $S$ is defined as follows:
\[ 
|g| = \mathsf{min}\{ k \in \mathbb{N} \, | \, g = s_1 \cdot \ldots \cdot s_k , \mbox{ where }s_i \in S\} . 
\] 
Let now $S\subseteq G$ be symmetric and $\overline{S}$ denote the the smallest conjugacy invariant subset of $G$ containing $S$. 
Assume that $G$ is generated by $\overline{S}$ but not necessarily by $S$. 
The {\em word norm} of an element $g \in G$ associated with $\overline{S}$ is defined by: 
\[ 
\parallel g\parallel = \mathsf{min}\{ k \in \mathbb{N}\, | \, g = s_1 \cdot \ldots \cdot s_k
, \mbox{ where }s_i \in \overline{S} \} .
\] 
This norm is conjugacy invariant. 
If $S$ is finite and $G$ is generated
by $\overline{S}$ then we say that $G$ is {\em normally finitely generated}. 
Note that if $(G \parallel \cdot \parallel )$ is such a group we have  $B_1 (1) =\overline{S} \cup\{ 1\}$ and $B_n (1)= (B_1 (1))^n$. 

In the case when $G$ is a free group of a group variety with a free base  
$x_0 , \ldots , x_i ,\ldots$, $i< \mathsf{n}$, where $\mathsf{n} \in \mathbb{N} \cup \{ \omega\}$, we will assume that 
\[ 
S=\{ x^{\pm 1}_0 , \ldots , x^{\pm 1}_i ,\ldots \, | \,  i< \mathsf{n} \}. 
\]
The free group $\mathsf{F_n}$, $\mathsf{n} \in \mathbb{N} \cup \{ \omega\}$, will be studied in the next subsection. 

\begin{remark} 
{\em Assume that $G$ is a LEF group and is normally generated by a finite $S$. 
Let $\parallel \cdot \parallel$ be the corresponding norm. 
Note that the following weak form of metrical LEF holds. }

For every finite subset $D\subseteq G$ there is a finite normed group $(C,\ell_C )$ and an injective map $\phi : D\rightarrow C$ 
such that 
\begin{itemize} 
\item any triple $h,g,hg\in D$ satisfies 
$\phi (hg )=\phi (h) \phi (g)$, 
\item  $\ell_C$ is a word norm generated by $\phi(S)$ and for any $g \in D$ we have 
\[ 
\ell_C (\phi (g)) \le \parallel g\parallel.  
\]
\end{itemize}  
{\em Indeed, for every $g\in D$ find its presentation as $g= s^{v_1}_1 \cdot \ldots \cdot s^{v_k}_{k}$ where 
$k=\parallel g \parallel$, $s_i \in S$. 
Extend $D\cup S$ by all $v_j$ and all sub-products appearing in these expressions. 
Let $D_1$ be the corresponding set. 
Using LEF find a finite group $C$ and an appropriate embedding $\phi : D_1 \to C$ with 
$C=\langle \phi(D_1) \rangle$. 
Let $\ell_C$ be the word norm generated by $\overline{\phi (S)}$.   
} 
\end{remark} 

\begin{remark} 
{\em 
Assume that $G$ is normally generated by finite $S$ and $S'$. 
Then there is a natural number $K$ such that 
each element of $S'$ is an $\overline{S}$-word of length $\le K$ and vice versa. 
The same property holds for $\overline{S}$ and $\overline{S}'$. 
As a result we see that for any $g\in G$, $\parallel g\parallel \le K \parallel g\parallel'$ and $\parallel g\parallel' \le K \parallel g\parallel$, i.e. $\parallel \cdot \parallel$ and $\parallel \cdot \parallel'$ are Lipschitz equivalent.  
}
\end{remark}

The following question looks very interesting. 
Assume that $G$, $\parallel \cdot \parallel$ and $\parallel \cdot \parallel'$ be as in this remark.   
Does the property that $G$ is metrically LEF  with respect to one of the norms $\parallel \cdot \parallel$ and $\parallel \cdot \parallel'$ imply the same property for another norm?   
We do not even know if these word norms generate the same topology. 
We guess that if it is so then the property that the group is metrically LEF with respect to one of the norms implies the same property for another one.  
In Section 3.2 we will see the weaker formulation that if two word norms generate the same topology, then the property that the group is metrically fully residually finite with respect to one of them implies the same property for another norm.  

\subsection{Free groups are not metrically LEF}

The following general problem looks very interesting. 
\begin{itemize} 
\item Assume that $G = \langle \overline{S} \rangle$ where $S$ is a finite symmetric subset of $G$. 
Is $(G \parallel \cdot \parallel )$ metrically LEF? 
\end{itemize} 
From now on in this section we concentrate on the case when $G$ is a free group. 
However in the following proposition we can make the situation slightly wider. 
Assume that $w(z_1,z_2,\ldots, z_m )$ is a group word and $V_w$ be the variety of groups defined by this word as an identity. 
By $\mathsf{F}^w_n$ we denote the free $n$-generated group of $V_w$. 
We will say that the word $w(\bar{z})$ is {\em quasi-linear} if for any group $H$ and its generating set $A$ the group $H$ belongs to $V_w$ if and only if $w(\bar{a})=1$ for each $n$-tuple $\bar{a}$ from $A$.  
Note that the trivial word is quasi-linear. 
Any multi-linear word (or an outer commutator identity) is quasi-linear too. 
In particular the variey of $\ell$-step nilpotent groups has the form $V_w$ for a quasi-linear $w(\bar{z})$. 

\begin{proposition} \label{f_g_1} 
Let $w(\bar{z})$ be a quasi-linear word and $\mathcal{C}$ be a class of normed groups such that their norms are integer valued and $\mathcal{C}$ is closed under subgroups. 

Then the free group $(\mathsf{F}^w_n, \parallel \cdot \parallel )$ 
(where $n$ is a natural number greater than $1$) is metrically LE$\mathcal{C}$ if and only if it is metrically fully residually $\mathcal{C}$.    
\end{proposition} 

\begin{proof} 
Sufficiency is obvious. 
Let us prove necessity. 
Let $D$ be a finite subset of $(\mathsf{F}_n, \parallel \cdot \parallel )$ and $Q$ be a finite subset of $\mathbb{N}$ with $0\in Q$. 
Let $\varphi :(\mathsf{F}^w_n, \parallel \cdot \parallel )\rightarrow (H, \ell ) \in \mathcal{C}$ be as in  Definition  
\ref{mLEF}. 
We may assume that there is $D_0 \subseteq D$ such that $S\subseteq D_0$, $D\subseteq \langle D_0 \rangle$ and for each sub-term $t(\bar{z})$ of $w(\bar{z})$ all values of the form $t(\bar{d})$ with $\bar{d}$ from $D_0$ belong to $D$.   
By quasi-linearity of $w(\bar{z})$ the set $\varphi (D  )$ generates a subgroup of $H$ which belongs to $V_w$. 
In particular we can extend $\varphi$ to a homomorphism into $H$, say $\hat{\varphi}$.  
Taking a subgroup if necessary we assume that $\hat{\varphi}$ is surjective. 

Since $S \subseteq D$, $\hat{\varphi} (S)\subseteq B^H_1 (1)$. 
The latter implies $\hat{\varphi} (\overline{S}) \subseteq B^H_1 (1)$ by invariantness of $\ell$. 
Now 
\[ 
\hat{\varphi} (B^{\mathsf{F}^w_n}_m (1)) = \hat{\varphi} (\overline{S}^m) = 
\hat{\varphi} (\overline{S})^m \subseteq B^H_m (1). 
\] 
Indeed, the latter inclusion follows from the triangle inequality.  
We see that $\ell (\hat{\varphi}(g)) \le \parallel g \parallel$ for any 
$g\in \mathsf{F}^w_n$. 
\end{proof}

The argument of the proof of this proposition gives the following lemma.

\begin{lemma} 
Let $(H, \ell )$ be a group with an invariant integer valued norm such that $B^H_n (1) \subseteq (B^H_1 (1))^n$ for all $n\in \omega$.  
Then $\ell$ is a word norm with respect to $B^H_1 (1)$. 
\end{lemma} 

\begin{proof} 
If $g \in (B^H_1 (1)\cdot B^H_1 (1) )\setminus B^H_1 (1)$, then $\ell (g) \ge 2$.  
By the triangle inequality we have $\ell (g) = 2$. 
In particular $B^H_2 (1) = (B^H_1 (1))^2$. 
If $g \in (B^H_1 (1) )^3 \setminus (B^H_1 (1))^2$, then $\ell (g) \ge 3$.
By the triangle inequality we have $\ell (g) = 3$. 
The rest follows by easy induction. 
\end{proof} 

Let us prove the following proposition.

\begin{proposition} \label{metrF_n}
A normally finitely generated group $G$ with a word norm $\parallel \cdot \parallel$ is metrically fully residually finite with respect to integer valued normed groups if and only if for any $m\in \mathbb{N}$ the set 
$B^G_m (1) =\{ g \in G \, | \, \parallel g \parallel \le m \}$ 
is closed in the profinite topology of $G$. 
\end{proposition}

\begin{proof} 
In this proof we assume that the norm is defined with respect to some finite symmetric $S$ with $G= \langle \overline{S} \rangle$. 
To see sufficiency of the proposition 
take any finite $K\subset G$ and  
$Q\subset \mathbb{N}$. 
We may assume that $Q$ is an initial segment of natural numbers and contains $\mathsf{max}\{ \parallel h \parallel \, | \, h\in K\} +1$. 
Since all $B^G_m (1)$ are closed for $m\in Q$, for each $g\in K\setminus \{ 1\}$ there is a subgroup $H_g \triangleleft G$ of finite index with 
$gH_g \cap B^{G}_{m_g} (1) = \emptyset$, where $m_g = \parallel g\parallel -1$. 
Now it is easy to find a subgroup of finite index of $G$, say $H_K$, such that 
for all $g\in K$ we have 
$gH_K \cap B^{G}_{m_g} (1) = \emptyset$. 
This obviously means that for each $g\in K$ the norm of $gH_K \in G/H_K$ with respect to $B^{G}_1 (1)H_K$ coincides with $\parallel g\parallel$. 
The rest is straightforward. 

Let us prove necessity. 
W.l.o.g. we may assume that there are elements $g\in G$ with 
$\parallel g \parallel \ge 2$. 
Let $(H,\ell )$ be a homomorphic image of $(G, \parallel \cdot \parallel )$ with respect to some $\varphi$. 
Let $D_{\ell} = B^H_1 (1) = \{ h \in H \, | \, \ell (h) \le 1 \}$.  
Then  $\varphi (\overline{S}) \subseteq D_{\ell}$. 

In order to prove that the set 
$B^G_m = \{ x \in G \, | \, \parallel x \parallel \le m \}$ is closed take any $g\not\in B^G_m$. 
Choose a homomorphism $\varphi$ from $(G \parallel \cdot \parallel )$ to some finite $(H,\ell )$ such that $\ell (\varphi (g)) = \parallel g \parallel$. 
If $\varphi (g) \in D^m_{\ell}$ then $\varphi (g) \in B^H_m (1)$ by the triangle inequality. 
Since $\parallel g \parallel >m$ we see that  
$\varphi (g) \not\in D^m_{\ell}$, i.e.  
$\varphi (g) \not\in \varphi(\overline{S}^m) = \varphi (B^G_m (1))$.  
\end{proof}

\begin{theorem} \label{F_2not}
The free group $(\mathsf{F}_2, \parallel \cdot \parallel )$  is not metrically LEF. 
\end{theorem} 

\begin{proof} 
Let $S= \{ x^{\pm 1}_1 , x^{\pm 1}_2 \}$. 
Note that $(\mathsf{F}_2, \parallel \cdot \parallel )$  is not metrically LEF with respect to the class of integer valued normed groups. 
Indeed, use Proposition \ref{f_g_1}, Proposition \ref{metrF_n} and an observation of D. Segal and J. Gismatullin (independently, see Remark 4 in \cite{NST}) and explanations of A. Thom in \cite{thomover}:  
there is a natural number $n$ such that $\overline{S}^n$ is not closed in the profinite topology. 
The latter is a consequence of a deep theorem of Nikolov and Segal from \cite{NS} (see Theorem 1.2 there). 

The proof of the theorem is finished by an application of Proposition \ref{RvsN}. 
\end{proof} 

\begin{remark}
{\em The following question looks very interesting. 
For $n>1$ consider $SL(n,\mathbb{Z})$ together with the word metric $\parallel \cdot \parallel$ associated with respect to the standard set of generators (i.e. $\{ 0,1\}$-transvections). } 
Is this group metrically LEF? 
\end{remark}

\subsection{When metric LEF does not depend on metric}

\begin{proposition} 
Assume that a group $G$ is normally generated by a finite set $S_1$ and also by a finite set $S_2$. 
Let $\parallel \cdot \parallel_1$ and $\parallel \cdot \parallel_2$ be the corresponding word norms. 
Assume that they generate the same topology.  Then $(G,\parallel \cdot \parallel_1 )$ is metrically fully residually finite with respect to normed groups if and only if so does $(G, \parallel \cdot \parallel_2 )$.  
\end{proposition}  

\begin{proof} 
Applying Proposition \ref{metrF_n} (and Proposition \ref{RvsN}) we see that 
if two word norms generate the same topology, then the property that the group with one of them is metrically fully residually finite with respect to normed groups implies the same property for another norm.  
\end{proof}

\section{A non-residually finite group which with a word norm is metrically LEF} 

Let $(G, \parallel \cdot \parallel )$ be a normed group with respect to a word norm 
(corresponding a finite set of generators). 
Let us consider the following table: 
\bigskip 

(i) $G$ is residually finite \, 
$\Rightarrow \, \, G$ is LEF,

(ii) $(G,\parallel \cdot \parallel)$ is metrically LEF \, $\Rightarrow \, \, G$ is LEF,

(iii) $G$ is residually finite \, $\not\Rightarrow \, \, (G, \parallel \cdot \parallel )$ is metrically LEF,

(iv) $(G, \parallel \cdot \parallel )$ is metrically LEF \, $\Rightarrow \, \, G$ is residually finite?  

\bigskip 

Statements (i) and (ii) are obvious. 
Statement (iii) follows from Theorem \ref{F_2not}. 
Item (iv) asks if a sronger version of (ii) holds. 
The goal of this section is to answer question (iv) in negative.  
The example which supports this answer is  so called $G_2$ from p. 49 of the book \cite{CSC}.
The construction is as follows. 
Let  
$$
H=\bigoplus_{i\in\mathbb{Z}}H_i,
$$
where for each $i\in\mathbb{Z}$ the group $H_i$ is a copy of the alternating group $A_5$. 
We define $\alpha\in \mathsf{Aut}(H)$ as the one-step shift given by: 
$$
\alpha( (h_i )_{i\in\mathbb{Z}})=(h_{i+1})_{i\in\mathbb{Z}},   
$$ 
where $(h_i )_{i\in\mathbb{Z}}$ represents an element from $H$ (i.e. the support of this sequence is finite).  
Let $\psi: \mathbb{Z} \to \mathsf{Aut}(H)$ be a homomorphism defined by $\psi(k)=\alpha^k$.
$$
\mbox{ Let } \, \, G_2 \, \mbox{ be the semidirect product } \, H\rtimes_\psi\mathbb{Z}. 
$$ 
Proposition 2.6.5 of \cite{CSC} states that $G_2$ is finitely generated but not residually finite. 
In fact $G_2$ is generated by the symmetric set $S= H_0 \cup \{ t^{\pm 1} \}$, where $t$ corresponds to the pair $(1_H , 1)$. 
Let $\parallel \cdot \parallel$ be the word norm corresponding to $\overline{S}$. 
Our aim is the following statement.  
\begin{quote} 
The normed group $(G_2,\parallel \cdot \parallel )$ is metrially LEF. 
\hspace{2cm} ($\dagger$)
\end{quote} 
In fact our result is more general. 
It relates to the question of  
which constructions in group theory preserve metrical LEF.  
We will see that some natural semidirect products have this property
\footnote{In the case of $(\dagger )$ we view $G_2$ as an extension of $(\mathbb{Z},\parallel \cdot \parallel )$.}.  

\subsection{The group $G_{\mathbb{Z}}$} 

We now describe conditions which will appear in the formulation of the main result.      
The following formula will appear in $(S4)$ below. 
$$
\Xi(u_1,u_2,u_3)=\exists x,y\big(u_3=x^{-1}u_2^{-1}xy^{-1}u_1^{-1}y\big)
$$  
$$ 
\hspace{2cm} 
(u_3 \mbox{ belongs to the product of conjugacy classes of } \, u^{-1}_2 \, \mbox{ and } u^{-1}_1). 
$$ 
Let $P$ be a finite group. 
We introduce the following statements concerning elements of $P$.    
\begin{enumerate} \label{a5}
        \item[(S1)] For any pair $a_1$ and $a_2$ the element $a_2$ is of the form  $a_2 =x^{-1}a_1^{-1}yxy^{-1}$ for some $x,y\in P$. 
        \item[(S2)] For any triple  $a_1,a_2,a_3$ there exist $u,v\in P$ 
such that 
$a_2=   a_3 u a^{-1}_1a^{-1}_3 vu^{-1}v^{-1}$.
        \item[(S3)] The product of any triple of non-trivial conjugacy classes in $P$ covers $P\setminus \{ 1_P\}$:  
$$
\forall u_1,u_2,u_3,u_4 \exists x,y,z
 \big( \bigwedge_{i =1}^{4}   u_i \not= 1 \to u_4=x^{-1}u_1xy^{-1}u_2yz^{-1}u_3z
\big).$$
        \item[(S4)] There are two non-trivial conjugacy classes such that their product does not cover $P\setminus \{ 1_P \}$: 
$$
\exists u_1,u_2,u_3  \big( \bigwedge_{i =1}^{3}   u_i \not= 1 \wedge \neg\ \Xi(u_1,u_2,u_3)\big).
$$
    \end{enumerate}
\begin{lemma}
Statements $(S1) - (S4)$ hold in $P= A_5$. 
\end{lemma}
\begin{proof}
This is verified by $\mathsf{Mathematica}$. 
We mention that it is well-known that statements $(S3)$ and $(S4)$ hold in     $A_5$. 
For example the triple $(a_1,a_2,a_3)=((12)(34),(12345),(12345))$ does not realize $\Xi(u_1,u_2,u_3)$.
\end{proof}

\begin{remark} 
{\em 
It is described in \cite{Malclom} when every element of $A_n$ can be written as a product of 2 or fewer elements of order $p$ (where $p\le n$). 
It is observed that this is possible only for alternating groups of sufficiently large degree.
For example $A_5$ does not have this property. }
\end{remark}

Now take any finite $P$ satisfying $(S1) - (S4)$ and repeat the construction of the beginning of this section replacing everywhere $A_5$ by $P$:  
$$
H=\bigoplus_{i\in\mathbb{Z}}H_i,
$$
where for each $i\in\mathbb{Z}$ the group $H_i$ is a copy of $P$. 
As before $\alpha\in \mathsf{Aut}(H)$ is the one-step shift: 
$$
\alpha( (h_i )_{i\in\mathbb{Z}})=(h_{i+1})_{i\in\mathbb{Z}},   
$$ 
where $(h_i )_{i\in\mathbb{Z}}$ represents an element from $H$.  
Let $\psi: \mathbb{Z} \to \mathsf{Aut}(H)$ be a homomorphism defined by $\psi(k)=\alpha^k$.
$$
\mbox{ Let } \, \, G_{\mathbb{Z}} \, \mbox{ be the semidirect product } \, H\rtimes_\psi\mathbb{Z}. 
$$ 
As before $G_{\mathbb{Z}}$ is generated by the symmetric set $S= H_0 \cup \{ t^{\pm 1} \}$, where $t$ corresponds to the pair $(1_H , 1)$ (see Proposition 2.6.5 of \cite{CSC}). 
Let $\parallel \cdot \parallel$ be the word norm over $\overline{S}$. 
The main theorem of this section is as follows. 

\begin{theorem} \label{GZmlef} 
The group $G_{\mathbb{Z}}$ is finitely generated and the normed group  $(G_{\mathbb{Z}},\parallel \cdot \parallel )$ is metrically LEF. 
In particular, so is $G_2$.  
\end{theorem} 
 
 Before the proof we give some preliminary material. 
We view elements of $H$ as (support) vectors:
$$
( h_i)_{i\in\mathbb{Z}} = \bar{h}=h_{i_1}\ldots h_{i_k} \, , \, i_j\in\mathbb{Z}, 
$$
after removing all $1_P$ in the sequence (sometimes we admit that some $h_{i_j}$ with $j\notin \{ 1,k\}$ are trivial).  
When $\bar{g_i}$ is an indexed element of $H$ we will denote it by:  
$$
\bar{g_i}=g_{i,i_1}\ldots g_{i,i_k} \, , \, i_j\in\mathbb{Z}.
$$
The expression $(h_{i})_k$ means that  $h_{i} \in H_i$ is taken to the $k$-th place. 
Let $t =(1_P ,1)\in G_{\mathbb{Z}}$;  
it corresponds to $\alpha$. 

The following statements are standard and easy (for example see the proof of Proposition 2.6.5 in \cite{CSC}, pp. 49 -- 50). 
\begin{itemize}
    \item Any element from $G_{\mathbb{Z}}$ can be uniquely written as $(\bar{h},k)=\bar{h}t^k$. 
    \item $t^k\bar{h}t^{-k}=\alpha^k(\bar{h})$. 
\item For any $(\bar{h},k)=\bar{h}t^k$ and   $(\bar{g},\ell )=\bar{g}t^l$ we have: 
$(\bar{h},k)(\bar{g},\ell ) = \bar{h}\alpha^k(\bar{g})t^{k+l}$.
\end{itemize}   

We will use the following notation. 
For $\mathsf{g}=\bar{g}t^{k}\in G_\mathbb{Z}$ 
with non-trial $\bar{g}$ we define $i_{min}(\bar{g})$ (resp. $i_{max}(\bar{g})$) as the smallest (resp. the largest) $i$ such that 
$g_i\neq 1$. 

This notation is also applied to elements of $H$. 
Then notice that for any $\bar{g}\in H\backslash\{ 1\}$:
\begin{align*}
i_{min}(\alpha(\bar{g}))&=i_{min}(\bar{g})-1,\\
i_{min}(\alpha^{-1}(\bar{g}))&=i_{min}(\bar{g})+1 ,
\end{align*}
and the same equations hold for $i_{max}$.

Let us introduce the following terms.  
\begin{itemize} 
    \item $T_+=\{\bar{g}\alpha(\bar{g}^{-1})t \, \in G_{\mathbb{Z}} \, : \bar{g}\in H\}$,
    \item $T_-=\{\alpha(\bar{g})\bar{g}^{-1}t^{-1}\, \in G_{\mathbb{Z}} \, : \bar{g}\in H\}$.
\end{itemize} 
{\bf Fact A.} 
{\em The conjugacy invariant closure $S^{G_{\mathbb{Z}}}$ is of the following form }
$$
\overline{S}=(\bigcup_{i\in\mathbb{Z}}H_i )\cup T_+\cup T_-,
$$ 
where by $\bigcup_{i\in\mathbb{Z}}H_i$ we denote the set of all elements of $H$ with single supports. 

\subsection{$[k,t]$-commutators and their properties}  
We now introduce a new notion. 
It will be basic in the proof of Theorem \ref{GZmlef}. 
 
\begin{definition} \label{commutators} 
Let $\bar{h}\in H, k\in\mathbb{Z}$. \\         
For $k>0$ we say that $\bar{h}$ is a $[k,t]$-{\bf commutator} if there exist $\bar{g}_1,\ldots,\bar{g}_k$ such that:
$$
\bar{h}=\bar{g}_1\alpha(\bar{g}_1^{-1})\ldots\bar{g}_k\alpha(\bar{g}_k^{-1}).
$$
For $k<0$ we say that $\bar{h}$ is a $[k,t]$-{\bf commutator} if there exist $\bar{g}_1,\ldots,\bar{g}_{k}$ such that:
        $$
        \bar{h}=\alpha(\bar{g}_1)\bar{g}_1^{-1}\ldots\alpha(\bar{g}_{k})\bar{g}_k^{-1}
$$
We say that $\bar{h}$ is a $[\pm,t]$-{\bf commutator} (i.e. $k=0$) if:
        \begin{itemize}
            \item $\bar{h}=\bar{g_1}\alpha(\bar{g_1}^{-1})\alpha(\bar{g_2})\bar{g_2}^{-1}$ or
            \item $\bar{h}=\alpha(\bar{g_1})\bar{g_1}^{-1}\bar{g_2}\alpha(\bar{g_2}^{-1})$, for some $\bar{g}_1,\bar{g}_2\in H$.
        \end{itemize}    
\end{definition} 

\begin{remark}
\em{
It is clear that for any $k>0$ (resp. $k<0$) if $\bar{h}$ is a $[k,t]$-commutator, then it is a $[k+1,t]$-commutator (resp. $[k-1,t]$-commutator). 
Furthermore, 
\begin{itemize} 
\item if $\bar{h}$ is a $[\pm1,t]$-commutator, then it is a $[\pm,t]$-commutator; 
\item 
if $\bar{g}t^{\pm1}\in T_+\cup T_-$, then $\bar{g}$ is a $[\pm1,t]$-commutator.
\end{itemize} 
}
\end{remark}

In Lemmas 
\ref{nsprz} - \ref{opo} we collect necessary information about $[k,t]$-commutators.  
It will be used in the proof of Theorem \ref{GZmlef}.

\begin{lemma}\label{nsprz}
For any $\bar{h}\in H$ there is $i\in \mathbb{Z}$ such that $\bar{h}$ can be decomposed into 
\begin{itemize} 
\item a product of a $[1,t]$-commutator and some $h^*\in H_i$ and 
\item a product of $[-1,t]$-commutator and some $g^*\in H_i$.
\end{itemize} 
\end{lemma}

\begin{proof} 
W.l.o.g. we assume $\bar{h}=h_1h_2\ldots h_k$, where some $h_{i}$ can be trivial. 
We want $\bar{g}=g_2g_3\ldots g_k\in H$ such that 
$\bar{h}=\bar{g}\alpha(\bar{g}^{-1})h^*$. 
Note: 
$$
    \alpha(\bar{g}^{-1})=(g_{2}^{-1})_1(g_{3}^{-1})_{2}\ldots(g_{k}^{-1})_{k-1},
$$
and 
\begin{align*}
        \bar{g}\alpha(\bar{g}^{-1})&=(g^{-1}_2)_1(g_2g_3^{-1})_{2}\ldots(g_{k-1}g_{k}^{-1})_{k-1}g_{k}.
\end{align*}
We pick $g_{2},g_{3},\ldots ,g_{k}$ in turn such that $g_{2}^{-1}=h_1$ and the equality $g_{j}g_{j+1}^{-1}=h_{j}$ holds  for $j=2,\ldots ,k-1$. 
Put $h^*=(g_{k}^{-1}h_{k})_{k}$. 
This realizes the statement of the lemma for $[1,t]$-commutators. 

The case of $[-1,t]$-commutators is similar. 
Note:
$$
    \alpha(\bar{g})=(g_2)_1(g_3)_2\ldots(g_k)_{k-1}.$$
and 
$$
    \alpha(\bar{g})\bar{g}^{-1}=(g_2)_1(g_3g_2^{-1})_2\ldots(g_kg_{k-1}^{-1})_{k-1}g_k^{-1}.
$$
In order to satisfy 
$\bar{h}=\alpha(\bar{g})\bar{g}^{-1}g^*$ pick $g_2,\ldots,g_k$ in turn such that $g_2=h_1$ and $g_{j+1}g_j^{-1}=h_j$ for $j=2,\ldots,k-1$. 
Put $g^*=(g_kh_k)_k$. 
\end{proof}

In order to illustrate operations in the group $H$ we use tables. 
\begin{example} 
{\em 
Let $\bar{g}=g_1g_2g_4$ and  
$\bar{h}=h_2h_3$. 
Then we have: 
\begin{center}
         \begin{tabular}{|c|c|c|c|c|}
\hline
                  Index & $1$ & $2$ & $3$ & $4$  \\ \hline
                 $\bar{g}$ & $g_1$ & $g_2$ &  & $g_4$  \\ \hline
                $\bar{h}$ & & $h_2$ & $h_3$ &   \\ \cline{1-5} 
                 Result & $g_1$ & $g_2h_{2}$ & $h_{3}$ & $g_4$   \\ \hline
\end{tabular}
     \end{center}
So $\bar{g}\cdot\bar{h}=g_1(g_2h_2)_2h_3g_4$.
} 
\end{example}

\begin{lemma}\label{2k}
    Any element from $H$ is both a $[2,t]$-commutator and a $[-2,t]$-commutator.
\end{lemma}

\begin{proof}
Take any $\bar{h} \in H$. 
W.l.o.g. we assume that 
$\bar{h}=h_1\ldots h_{\ell}$ and some $h_{i}$ can be trivial. 
To show that $\bar{h}$ is a $[2,t]$-commutator we build by induction some $\bar{g}_1$ and $\bar{g}_2$ (starting with position 2) such that the equality 
$\bar{g}_1\alpha(\bar{g}_1^{-1})\bar{g}_2\alpha(\bar{g}_2^{-1})=\bar{h}$ 
holds.  
In order to find positions 2,3 of elements $\bar{g}_1,\bar{g}_2$ consider the following table. 
\begin{center}
\begin{tabular}{|c|c|c|c|c}
\cline{1-4}
                 Index & $1$ & $2$ & $3$ &  \\ \cline{1-4}
                $\bar{g_1}$ &  & $x^{-1}h_1^{-1}$ & $y^{-1}$ &  \\ \cline{1-4}
                 $\alpha(\bar{g_1}^{-1})$ & $h_1x$ & $y$ &  &  \\ \cline{1-4}
                $\bar{g_2}$ &  & $x$ & $y$ &  \\ \cline{1-4}
                 $\alpha(\bar{g_2}^{-1})$ & $x^{-1}$ & $y^{-1}$ &  &  \\ \cline{1-4}
                 Result & $h_1$ & $x^{-1}h_1^{-1}yxy^{-1}$ & $id$ &  \\ \cline{1-4}
\end{tabular}
\end{center} 
By (S1) of Section 5.1 we can pick $x,y$ such that putting:
\begin{align*}
    g_{1,2}&:=x^{-1}h_1^{-1}, & g_{1,3}&:=y^{-1},\\
    g_{2,2}&:=x,& g_{2,3}&:=y,
\end{align*}
the equality 
$\bar{g}_1\alpha(\bar{g}_1^{-1})\bar{g}_2\alpha(\bar{g}_2^{-1})=\bar{h}$ 
holds for positions 1 and 2 of $\bar{h}$. 

At the next step we extend $\bar{g}_1$ i $\bar{g}_2$ to positions 4 and 5. 
We will use property (S2) of Section 5.1 in a slightly modified form:   
\begin{quote} 
for any triple  $a_1,a_2,a_3$ there exist $z,u,v\in P$ 
such that  
$a_1 = a^{-1}_3 z a_3 u$ and 
$a_2= z^{-1}vu^{-1}v^{-1}$.
\end{quote} 
Now consider 
\begin{center}
    \begin{tabular}{|c|c|c|c|c|c|}
\hline
                 Index & $1$ & $2$ & $3$ & $4$ & $5$ \\ \hline
                $\bar{g_1}$ &  & $x^{-1}h_1^{-1}$ & $y^{-1}$ & $z^{-1}$ & $v^{-1}$ \\ \cline{1-6} 
                 $\alpha(\bar{g_1}^{-1})$ & $h_1x$ & $y$ & $z$ & $v$ &  \\ \hline
                $\bar{g_2}$ &  & $x$ & $y$ & $u^{-1}$ & $v$ \\ \cline{1-6} 
                 $\alpha(\bar{g_2}^{-1})$ & $x^{-1}$ & $y^{-1}$ & $u$ & $v^{-1}$ &  \\ \hline
                 Result & $h_1$ & $h_2$ & $y^{-1}zyu$ & $z^{-1}vu^{-1}v^{-1}$ & $id$ \\ \hline
\end{tabular}
\end{center} 
where $x$ and $y$ take the already found values.  
Using (S2) in the form above we can pick $z,u,v$ such that putting:
\begin{align*}
    g_{1,4}&:=z^{-1}, & g_{1,5}&:=v^{-1},\\
    g_{2,4}&:=u^{-1},& g_{2,5}&:=v,
\end{align*} 
the equality 
$\bar{g}_1\alpha(\bar{g}_1^{-1})\bar{g}_2\alpha(\bar{g}_2^{-1})=\bar{h}$ 
holds for positions 1,2,3 and 4 of $\bar{h}$. 

Using (S2) we repeat the latter step 
$\frac{\ell}{2}-2$ times for even $\ell$  
and 
$[\frac{\ell}{2}]-1$ times for odd $\ell$.   
As a result the obtained $\bar{g}_1$ and $\bar{g}_2$ satisfy the equality  $\bar{g}_1\alpha(\bar{g}_1^{-1})\bar{g}_2\alpha(\bar{g}_2^{-1})=\bar{h}$.

To show that $\bar{h}$ is a $[-2,t]$-commutator we have to find $\bar{g}'_1$ and $\bar{g}'_2$ satisfying the equality 
$\alpha(\bar{g}'_1 )(\bar{g}'_1)^{-1}\alpha(\bar{g}'_2)(\bar{g}'_2)^{-1}=\bar{h}$. 
In order to build $\bar{g}'_1,\bar{g}'_2$ we use a similar inductive procedure.   

To find positions $\ell-1$ and $\ell$ consider the following table. 
\begin{center}
\begin{tabular}{|c|c|c|c|c}
\cline{1-4}
                 Index & $\ell-2$ & $\ell-1$ & $\ell$ &  \\ \cline{1-4}
                $\alpha(\bar{g}'_1)$ & $y^{-1}$ & $x^{-1}h_{\ell}^{-1}$ &  &  \\ \cline{1-4}
                 $(\bar{g}'_1)^{-1}$ &  & $y$ & $h_{\ell}x$ &  \\ \cline{1-4}
                $\alpha(\bar{g}'_2)$ & $y$ & $x$ &  &  \\ \cline{1-4}
                 $(\bar{g}'_2)^{-1}$ &  & $y^{-1}$ & $x^{-1}$ &  \\ \cline{1-4}
                 Result & $id$ & $x^{-1}h_{\ell}^{-1}yxy^{-1}$ & $h_{\ell}$ &  \\ \cline{1-4}
\end{tabular}
\end{center} 
Applying (S1) find $x,y$ such that putting:
\begin{align*}
    g'_{1,i-1}&:=y^{-1}, & g'_{1,i}&:=x^{-1}h_i^{-1},\\
    g'_{2,i-1}&:=y,& g_{2,i}&:=x,
\end{align*}
the equality 
$\alpha(\bar{g}'_1 )(\bar{g}'_1)^{-1}\alpha(\bar{g}'_2)(\bar{g}'_2)^{-1}=\bar{h}$ 
holds for positions $\ell-1$ and $\ell$. 
The rest of the procedure is based on consecutive application of (S2) as in the first part of the proof. 
\end{proof}

\begin{lemma}\label{pmk3}
Let  $\bar{h}=h_1h_2h_3\in H$ with non-trivial $h_i$. 
Then $\bar{h}$ is a $[\pm,t]$-commutator if and only if the statement $\Xi (h_1,h_2,h_3)$ (from Section 4.1) is satisfied in $P$.
\end{lemma}

\begin{proof}
The element $\bar{h}=h_1h_2h_3$ is a $[\pm,t]$-commutator if and only if there exist $\bar{g}_1=abe, \bar{g}_2=cde$ such that:
 $$
\left\{
\begin{array}{l}
h_1=ac^{-1}\\
h_2=ba^{-1}cd^{-1}\\
h_3=eb^{-1}de^{-1} .
\end{array}
\right.
$$
This system of equations is equivalent to:
$$
\left\{
\begin{array}{l}
c^{-1}=a^{-1}h_1\\
d^{-1}=a^{-1}h_1ab^{-1}h_2\\
h_3=eb^{-1}h_2^{-1}ba^{-1}h_1^{-1}ae^{-1} .
\end{array}
\right.
$$
This system has a solution in $P$ if and only if property $\Xi(h_1,h_{2},h_{3})$ is satisfied. 
\end{proof}

The following lemma extends 
the sufficiency of Lemma \ref{pmk3}
to all elements of support of size 3.

\begin{lemma}\label{pos} 
Let $i_1<\ldots<i_\ell$ be any sequence of integer numbers and let 
$\bar{h}=h_1\ldots h_\ell \in H$ 
be a $[k,t]$-commutator 
(resp. $[\pm,t]$-commutator).  

Then $\bar{h}^*=(h_1)_{i_1}\ldots(h_\ell)_{i_\ell}$ is a $[k,t]$-commutator (resp. $[\pm,t]$-commutator).
\end{lemma}

\begin{proof}
To prove the lemma we introduce partial shift transformations $\beta_{i,s}: H\to H$. 
For $\bar{g}=g_kg_{k+1}\ldots g_{\ell}$  and $k\le i\le \ell$ let:
$$
\beta_{i,s}(\bar{h}):=h_k\ldots h_i(h_i)_{i+1}\ldots(h_i)_{i+s}(h_{i+1})_{i+s+1}\ldots(h_{\ell})_{\ell+s}.
$$

It is easy to see that  
\begin{align*}
\bar{g}\alpha(\bar{g}^{-1})&=(g_k^{-1})_{k-1}(g_kg_{k+1}^{-1})_k\ldots(g_{\ell-1}g_\ell^{-1})_{\ell-1}g_\ell,\\ 
\alpha(\bar{g})\bar{g}^{-1}&=(g_k)_{k-1}(g_{k+1}g_k^{-1})_k\ldots(g_\ell g_{\ell-1}^{-1})_{\ell-1} g_\ell^{-1}, 
\end{align*} 
and 
 
$\beta_{i,s}(\bar{g})\alpha(\beta_{i,s}(\bar{g})^{-1}) =(g_k^{-1})_{k-1}(g_kg_{k+1}^{-1})_k\ldots(g_{i-1}g_i^{-1})_{i-1}(g_ig_{i+1}^{-1})_{i+s}\ldots$

\hspace{9cm} $(g_{\ell-1} g_\ell^{-1})_{\ell+s-1}(g_\ell)_{\ell+s}$,  

$\alpha(\beta_{i,s}(\bar{g}))\beta_{i,s}(\bar{g})^{-1} =(g_k)_{k-1}(g_{k+1}g_k^{-1})_k\ldots(g_{i}g_{i-1}^{-1})_{i-1}(g_{i+1}g_i^{-1})_{i+s}\ldots$ 

\hspace{9cm} $(g_\ell g_{\ell-1}^{-1})_{\ell+s-1}(g_\ell^{-1})_{\ell+s}.$ 

Thus in order to write the $[\pm1,t]$-commutators generated by $\beta_{i,s}(\bar{g})$ one should take the $[\pm1,t]$-commutator generated by $\bar{g}$ and then shift it to the right by $s$ positions starting with the index $i$. 

On the other hand a $[\pm n,t]$-commutator is a product of $n$ $[\pm1,t]$-commutators and a $[\pm,t]$-commutator is a product of $[1,t]$-commutator and $[-1,t]$-commutator. 
Thus the same idea can be applied to these objects too. 
In particular, if 
$\bar{g}_1,\ldots,\bar{g}_{s} \in H$ form  $\bar{h}$ (from the formulation) as in Definition \ref{commutators} (for $[\pm,t]$-commutator $s=2$), then  transforming each $\bar{g}_i$ as follows 
$$
\bar{g}'_i:=\alpha^{-i_1+1}\circ\beta_{2,i_2-i_1-1}\circ\ldots\circ\beta_{\ell,i_\ell-i_{\ell-1}-1}(\bar{g}_i), 
$$
and replacing each $\bar{g}_i$ by $\bar{g}'_i$ in the corresponding term (see Definition \ref{commutators}) we obtain $(h_1)_{i_1}\ldots(h_\ell)_{i_\ell}$.
\end{proof}

The opposite implication is also true but  it requires some additional effort (see Lemma \ref{opo}).

\begin{lemma}\label{pmk4}
If an element of $H$ has support of size $\ge 4$, then it is a $[\pm,t]$-commutator.
\end{lemma}
\begin{proof}
     By Lemma \ref{pos} it suffices to consider the case when $\bar{h}=h_1\ldots h_{\ell}$.
     
Let $\ell = 4$, i.e. $\bar{h}=h_1h_2h_3h_4$ with non-trivial $h_i$. 
To show that $\bar{h}$ is $[\pm,t]$-commutator we build the required $\bar{g}_1$ and $\bar{g}_2$ on positions 2,3,4 such that the product 
$\alpha(\bar{g_1})\bar{g_1}^{-1}\bar{g_2}\alpha(\bar{g_2}^{-1})$ has the following form:
     \begin{center}
         \begin{tabular}{|c|c|c|c|c|}
\hline
                  Index & $1$ & $2$ & $3$ & $4$  \\ \hline
                $\alpha(\bar{g_1})$ & $a$ & $b$ & $c$ &    \\ \cline{1-5} 
                 $\bar{g_1}^{-1}$ &  & $a^{-1}$ & $b^{-1}$ & $c^{-1}$  \\ \hline
                $\bar{g_2}$ &  & $h_1^{-1}a$ & $h_{2}^{-1}ba^{-1}h_1^{-1}a$ &  $h_{3}^{-1}cb^{-1}h_{2}^{-1}ba^{-1}h_1^{-1}a$ \\ \cline{1-5} 
                  $\alpha(\bar{g_2}^{-1})$& $a^{-1}h_1$ & $a^{-1}h_1ab^{-1}h_{2}$ & $a^{-1}h_1ab^{-1}h_{2}bc^{-1}h_{3}$ &    \\ \hline
                 Result & $h_1$ & $h_{2}$ & $h_{3}$ & $c^{-1}h_{3}^{-1}cb^{-1}h_{2}^{-1}ba^{-1}h_1^{-1}a$   \\ \hline
\end{tabular}
     \end{center}
By (S3) of Section 5.1 we can pick $a,b,c,$ such that $\alpha(\bar{g}_1)\bar{g}_1^{-1}\bar{g}_2\alpha(\bar{g}_2^{-1})=h_1h_{2}h_{3}h_{4}$.
     
Now let $\bar{h}=h_1h_2h_3h_4\ldots h_{\ell}$ be of support of size $\ell > 4$. 
In order to show that $\bar{h}$ is a $[\pm,t]$-commutator note that  
$\bar{g}_1=a_2a_3\ldots a_{\ell}$ and  
$\bar{g}_2=b_2b_3\ldots b_{\ell}$ satisfy  
$\alpha(\bar{g_1})\bar{g_1}^{-1}\bar{g_2}\alpha(\bar{g_2}^{-1})=\bar{h}$ if and only if  the following system is satisfied too:
$$
\left\{
\begin{array}{l}
h_1=a_2b_2^{-1}\\
h_2=a_3a_2^{-1}b_2b_3^{-1}\\
\ldots\\
h_{\ell-1}=a_{\ell}a_{\ell-1}^{-1}b_{\ell-1}b_{\ell}^{-1}\\
h_{\ell}=a_{\ell}^{-1}b_{\ell} .
\end{array}
\right.
$$
The latter one is equivalent to:
\begin{equation}\label{ul55}
\left\{
\begin{array}{l}
b_2=(h_1)^{-1}a_2^{-1}\\
b_3=(h_2)^{-1}a_3a_2^{-1}(h_1)^{-1}a_2\\
\ldots\\
b_{\ell}=(h_{\ell-1})^{-1}a_{\ell} a_{\ell -1}^{-1}(h_{\ell -2})^{-1}a_{\ell -1}\ldots a_2^{-1}(h_1)^{-1}a_2\\
h_{\ell}=a_{\ell}^{-1}(h_{\ell-1})^{-1}a_{\ell}a_{\ell-1}^{-1}(h_{\ell -2})^{-1}a_{\ell -1}\ldots a_2^{-1}(h_1)^{-1}a_2 . 
\end{array}
\right.
\end{equation}
We now show, that we can pick $a_i, i=2,3,\ldots,\ell$, such that:
$$
     a_\ell^{-1}(h_{\ell -1})^{-1}a_\ell a_{\ell -1}^{-1}(h_{\ell -2})^{-1}a_{\ell -1}\ldots a_2^{-1}(h_1)^{-1}a_2=h_\ell .
$$ 
Consider the following equation: 
$$
x_\ell^{-1}(h_{\ell -1})^{-1}x_\ell x_{\ell -1}^{-1}(h_{\ell -2})^{-1}x_{\ell -1}\ldots x_2^{-1}(h_1)^{-1}x_2=h_\ell .
$$ 
Since $h_\ell \not= 1 \not=h_{\ell -1}$ there is $a_\ell$ such that 
$a_\ell^{-1}h_{\ell -1}a_\ell \neq (h_{\ell})^{-1}$ (for example, this is a consequence of (S3)). 
Putting $x_\ell :=a_\ell$ and multiplying from the left both sides of this equation by $a_\ell^{-1}h_{\ell -1}a_\ell$ we have:
$$
x_{\ell-1}^{-1}(h_{\ell -2})^{-1}x_{\ell -1}\ldots x_2^{-1}(h_1)^{-1}x_2=a_\ell^{-1}h_{\ell-1}a_\ell h_\ell\neq 1.
     $$
     Repeating this operation $\ell -5$ times we obtain the following equation. 
$$
x_4^{-1}h^{-1}_3 x_4 x_3^{-1}h^{-1}_2 x_3x_2^{-1}h^{-1}_1x_2=a_5^{-1}h_4a_5\ldots a_\ell^{-1}h_{\ell -1}a_\ell h_\ell \mbox{ (which is not 1)}.
     $$
By (S3) there are $a_1,a_2,a_3,a_4$ realizing the equation. 
Then using system \eqref{ul55} we calculate $b_2,\ldots,b_\ell$ and eventually $\bar{g}_1,\bar{g}_2$ which satisfy the equation $\alpha(\bar{g_1})\bar{g_1}^{-1}\bar{g_2}\alpha(\bar{g_2}^{-1})=\bar{h}$. 
This proves the Lemma.
     \end{proof}

\begin{remark}\label{poz}
{\em Note that in Lemma \ref{pmk4} supports of $\bar{g}_1$ and $\bar{g}_2$ consist of indexes of $\bar{h}$ without the first one.
} 
\end{remark}

We now show the opposite implication of Lemma \ref{pos}.

\begin{lemma}\label{opo}
    Let $i_1<\ldots<i_\ell$ be an increasing sequence of integers and  
$\bar{h}=h_{i_1}\ldots h_{i_\ell}$ be a $[k,t]$-commutator (resp. $[\pm,t]$-commutator) with non-trivial $h_{i_j}$. 
Then $\bar{h}'=(h_{i_1})_1\ldots (h_{i_\ell})_\ell$ is a $[k,t]$-commutator (resp. $[\pm,t]$-commutator) too. 
\end{lemma}

\begin{proof}
By Lemmas \ref{2k} and \ref{pmk4} the statement of the lemma holds for $[k,t]$-commutators with $|k|\geq2$ and for $[\pm,t]$-commutators under the assumption $\ell\geq4$. 
So we only need to verify the case of a $[\pm1,t]$-commutator and the cases of $[\pm,t]$-commutators with $\ell=2,3$.

Suppose that $h_{i_1}\ldots h_{i_\ell}$ is a $[1,t]$-commutator. 
Then there exist 
$\bar{g}=g_{i_1+1}g_{i_1+2}\ldots g_{i_\ell}$ such that the following equations hold:
$$
\left\{
\begin{array}{l}
h_{i_1}=g_{i_1+1}^{-1}\\
1=g_{i_1+1}g_{i_1+2}^{-1}\\
\ldots\\
1=g_{i_2-1}g_{i_2}^{-1}\\
h_{i_2}=g_{i_2}g_{i_2+1}^{-1}\\
1=g_{i_2+1}g_{i_2+2}^{-1}\\
\ldots\\
1=g_{i_\ell-1}g_{i_\ell}^{-1}\\
h_{i_\ell}=g_{i_\ell}
\end{array}
\right.
$$
This system of equations is equivalent to:
$$
\left\{
\begin{array}{l}
g_{i_1+1}^{-1}=h_{i_1}\\
g_{i_1+2}^{-1}=h_{i_1}\\
\ldots\\
g_{i_2}^{-1}=h_{i_1}\\
g_{i_2+1}^{-1}=h_{i_1}h_{i_2}\\
g_{i_2+2}^{-1}=h_{i_1}h_{i_2}\\
\ldots\\
g_{i_\ell}^{-1}=h_{i_1}\ldots h_{i_{\ell-1}}\\
h_{i_\ell}=(h_{i_1}\ldots h_{i_{\ell-1}})^{-1}
\end{array}
\right.
$$
Let $\bar{g}^*=(h_{i_1}^{-1})_2((h_{i_1}h_{i_2})^{-1})_3\ldots((h_{i_1}\ldots h_{i_{\ell-1}})^{-1})_\ell$. 
Then we have:
\begin{align*}
    \bar{g}^*\alpha(\bar{g}^{*-1})&=(h_{i_1})_1(h_{i_1}^{-1}h_{i_1}h_{i_2})_2\ldots ((h_{i_1}\ldots h_{i_{\ell-1}})^{-1})_\ell\\
    &=(h_{i_1})_1(h_{i_2})_2\ldots (h_{i_\ell})_\ell . 
\end{align*}
This finishes the case of a $[1,t]$-commutator.  
The case of a $[-1,t]$-commutator is analogous.

Let us consider the case of a $[\pm,t]$-commutator and $\ell=2$. 
Let $\bar{h}=h_{i_1}h_{i_2}$. 
Since $\bar{h}$ is a $[\pm,t]$-commutator then there are $\bar{a}=a_{i_1+1}\ldots a_{i_2}$ and  $\bar{b}=b_{i_1+1}\ldots b_{i_2}$ such that:
$$
\alpha(\bar{a})\bar{a}^{-1}\bar{b}\alpha(\bar{b}^{-1})=h_{i_1}h_{i_2}.
$$ 
This equality implies the following system of equations: 
$$
\left\{
\begin{array}{l}
h_{i_1}=a_{i_1+1}b_{i_1+1}^{-1}\\
1=a_{i_1+2}a_{i_1+1}^{-1}b_{i_1+1}b_{i_1+2}^{-1}\\
\ldots\\
1=a_{i_2}a_{i_2-1}^{-1}b_{i_2-1}b_{i_2}^{-1}\\
h_{i_2}=a_{i_2}^{-1}b_{i_2} .
\end{array}
\right.
$$
The latter is equivalent to: 
$$
\left\{
\begin{array}{l}
a_{i_1+1}=h_{i_1}b_{i_1+1}\\
a_{i_1+2}=b_{i_1+2}b_{i_1+1}^{-1}h_{i_1}b_{i_1+1}\\
a_{i_1+3}=b_{i_1+3}b_{i_1+1}^{-1}h_{i_1}b_{i_1+1}\\
\ldots\\
a_{i_2}=b_{i_2}b_{i_1+1}^{-1}h_{i_1}b_{i_1+1}\\
h_{i_2}=b_{i_1+1}^{-1}h_{i_1}^{-1}b_{i_1+1} . 
\end{array}
\right.
$$ 
This implies that 
$\bar{h}^*=(h_{i_1})_1(h_{i_2})_2$ is a $[\pm,t]$-commutator  
$$
\alpha((a_{i_1+1} )_2) (a^{-1}_{i_1+1})_2 (b_{i_1+1})_2 \alpha( (b_{i_1+1}^{-1})_2).
$$

In the case of a $[\pm,t]$-commutator and $\ell=3$ there are $\bar{a}=a_{i_1+1}\ldots a_{i_3}, \bar{b}=b_{i_1+1}\ldots b_{i_3}$ such that:
$$
\alpha(\bar{a})\bar{a}^{-1}\bar{b}\alpha(\bar{b}^{-1})=h_{i_1}h_{i_2}h_{i_3}.
$$ 
It is equivalent to the following system of equations.  
$$
\left\{
\begin{array}{l}
h_{i_1}=a_{i_1+1}b_{i_1+1}^{-1}\\
1=a_{i_1+2}a_{i_1+1}^{-1}b_{i_1+1}b_{i_1+2}^{-1}\\
\ldots\\
1=a_{i_2}a_{i_2-1}^{-1}b_{i_2-1}b_{i_2}^{-1}\\
h_{i_2}=a_{i_2+1}a_{i_2}^{-1}b_{i_2}b_{i_2+1}^{-1}\\
1=a_{i_2+2}a_{i_2+1}^{-1}b_{i_2+1}b_{i_2+2}^{-1}\\
\ldots\\
1=a_{i_3}a_{i_3-1}^{-1}b_{i_3}b_{i_3-1}^{-1}\\
h_{i_3}=a_{i_3}^{-1}b_{i_3}
\end{array}
\right.
$$ 
The latter one is equivalent to: 
$$
\left\{
\begin{array}{l}
a_{i_1+1}=h_{i_1}b_{i_1+1}\\
a_{i_1+2}=b_{i_1+2}b_{i_1+1}^{-1}h_{i_1}b_{i_1+1}\\
a_{i_1+3}=b_{i_1+3}b_{i_1+1}^{-1}h_{i_1}b_{i_1+1}\\
\ldots\\
a_{i_2}=b_{i_2}b_{i_1+1}^{-1}h_{i_1}b_{i_1+1}\\
a_{i_2+1}=h_{i_2}b_{i_2+1}b_{i_1+1}^{-1}h_{i_1}b_{i_1+1}\\
a_{i_2+2}=b_{i_2+2}b_{i_2+1}^{-1}h_{i_2}b_{i_2+1}b_{i_1+1}^{-1}h_{i_1}b_{i_1+1}\\
a_{i_2+3}=b_{i_2+3}b_{i_2+1}^{-1}h_{i_2}b_{i_2+1}b_{i_1+1}^{-1}h_{i_1}b_{i_1+1}\\
\ldots\\
a_{i_3}=b_{i_3}b_{i_2+1}^{-1}h_{i_2}b_{i_2+1}b_{i_1+1}^{-1}h_{i_1}b_{i_1+1}\\
h_{i_3}=b_{i_1+1}^{-1}h_{i_1}^{-1}b_{i_1+1}b_{i_2+1}^{-1}h_{i_2}^{-1}b_{i_2+1}
\end{array}
\right.
$$
We see that $\Xi(h_{i_1},h_{i_2},h_{i_3})$ is satisfied. 
By Lemma \ref{pmk3} the element 
$(h_{i_1})_1(h_{i_2})_2(h_{i_3})_3$ is a $[\pm,t]$-commutator.
\end{proof}

\subsection{The word norm of $G_{\mathbb{Z}}$}

Now we are ready to describe the word norm of $G_\mathbb{Z}$ with respect to generators $S=H_0\cup\{t,t^{-1}\}$.
We remind the reader that:
$$
\overline{S}=(\bigcup_{i\in\mathbb{Z}}H_i )\cup T_+\cup T_-,
$$
where $T_+=\{g\alpha(g^{-1})t: g\in H\}$ and $T_-=\{\alpha(g)g^{-1}t^{-1}: g\in H\}$ (see Fact A in Section 5.1).  

Below we will use the natural weight function $\omega: H\to \mathbb{N}$ defined by $\omega(\bar{h})= |\mathsf{supp}(\bar{h})|$. 

\begin{lemma}\label{ng1}
The invariant word norm of $G_\mathbb{Z}$ generated by $\overline{S}$ is of the following form: 
$$
||\bar{h}t^n||=\left\{
\begin{array}{cl}
0&\mbox{if } n=0 \mbox{ and } \bar{h}=1,\\
1&\mbox{if } n=0 \mbox{ and } \bar{h}=h_i,\\
2&\mbox{if } n=0 \mbox{ and } \omega(\bar{h})=2, \\
2&\mbox{if } n=0 \mbox{ and } \bar{h}\mbox{ is a } [\pm,t]\mbox{-commutator},\\
3&\mbox{if } n=0,\ \omega(\bar{h})=3 \mbox{ and }\bar{h} \mbox{ is not a }[\pm,t]\mbox{-commutator,}\\
1&\mbox{if } n=\pm1 \mbox{ and } \bar{h} \mbox{ is a }[\pm1,t]\mbox{-commutator},\\
2&\mbox{if } n=\pm1 \mbox{ and } \bar{h} \mbox{ is not a }[\pm1,t]\mbox{-commutator},\\
|n|&\mbox{if } |n|>1.
\end{array}
\right.
$$
\end{lemma}

\begin{proof}
Case 1: $\mathsf{h}=\bar{h}t^0=\bar{h}\in G_\mathbb{Z}$. 

Obviously if $\bar{h}=1$, then $||\bar{h}||=0$. Let $\bar{h}\neq 1$. 
When $\omega(\bar{h})=1$, we have $\bar{h}=h_i\in \overline{S}$ and $||\bar{h}||=1$. 

If $\omega(\bar{h})=2$, then $\bar{h}$ is of the form $\bar{h}=h_{i_1}h_{i_2}$, where $h_{i_1},h_{i_2}\in \overline{S}$. 
Since $\bar{h}\notin \overline{S}$, then $||\bar{h}||=2$.

Now let $\omega(\bar{h})=3$ and $\bar{h}=h_{i_1}h_{i_2}h_{i_3}$. 
If $\Xi(h_{i_1}^{-1},h_{i_2}^{-1},h_{i_3}),$ is satisfied, then by Lemma \ref{pmk3}:
\begin{equation*}
   \bar{h}=\alpha(\bar{g})\bar{g}^{-1}\bar{g}'\alpha(\bar{g}'^{-1}).
\end{equation*}
Note that 
$\alpha(\bar{g})\bar{g}^{-1}t^{-1}$ and $\alpha(\bar{g}'\alpha(\bar{g}'^{-1}))t$ belong to $\overline{S}$. 
Furthermore:
\begin{align*}
    \alpha(\bar{g})\bar{g}^{-1}t^{-1}\alpha(\bar{g}'\alpha(\bar{g}'^{-1}))t&=\alpha(\bar{g})\bar{g}^{-1}\alpha^{-1}(\alpha(\bar{g}'\alpha(\bar{g}'^{-1})))=\\
    &=\alpha(\bar{g})\bar{g}^{-1}\bar{g}'\alpha(\bar{g}'^{-1})=\bar{h}.
\end{align*}
We see that $\bar{h}$ is a product of two elements from $\overline{S}$, i.e. $||\bar{h}||=2$.

If $\Xi(h_{i_1}^{-1},h_{i_2}^{-1},h_{i_3})$ is not satisfied, then by Lemma \ref{pmk3} the element $\bar{h}$ is not a $[\pm,t]$ commutator. 
Thus $\bar{h}$ cannot be presented neither as a product from $(T_+)(T_-) \cup (T_-)(T_+)$ nor a product from 
$((\bigcup_{i\in\mathbb{Z}}H_i)(T_+\cup T_-))^{\pm1}$. 
In particular, $\bar{h}$ cannot be written as product of two elements from $\overline{S}$. 
Since $\bar{h}=h_{i_1}h_{i_2}h_{i_3}$ 
where $h_{i_1},h_{i_2},h_{i_3}\in \overline{S}$, we conclude that $||\bar{h}||=3$.

When $\omega(\bar{h})>3$, we apply  the argument above for $[\pm ,t]$-commutators together with Lemma \ref{pmk4}. 
Then we see that $||\bar{h}||=2$.

Case 2: $\mathsf{h}=\bar{h}t^{\pm1}$

If $\mathsf{h}\in T_+\cup T_-$ (i.e. $\bar{h}$ is a $[\pm1,t]$-commutator), then obviously $||\mathsf{h}||=1$. 
For $\mathsf{h}\notin T_+\cup T_-$ by Lemma \ref{nsprz} we can decompose $\bar{h}$ as $\bar{a}\cdot a^*$ or $\bar{b}\cdot b^*$, where $\bar{a}$ is a $[1,t]$-commutator, $\bar{b}$ is a $[-1,t]$-commutator and $a^*,b^*\in\bigcup_{i\in\mathbb{Z}}H_i$. 
Then we have one of the following cases. 
\begin{align*}
    \mathsf{h}&=\bar{h}t=\bar{a}a^*t=\bar{a}tt^{-1}a^*t=\bar{a}t\alpha^{-1}(a^*),\\
    \mathsf{h}&=\bar{h}t^{-1}=\bar{b}b^*t^{-1}=\bar{b}t^{-1}tb^*t^{-1}=\bar{b}t^{-1}\alpha(b^*).
\end{align*}
We see that $\bar{a}t\in T_+$, $\bar{b}t^{-1}\in T_-$ and $\alpha^{-1}(a^*),\alpha(b^*)\in\bigcup_{i\in\mathbb{Z}}H_i$, i.e. $||\mathsf{h}||=2$.

Case 3: $\mathsf{h}=\bar{h}t^k, |k|>1$

Suppose that $k>1$. 
Obviously $||\mathsf{h}||\geq k$. 
By Lemma \ref{2k} we know that $\bar{h}$ is a $[2,t]$-commutator. 
Thus there exist $\mathsf{g}_1,\mathsf{g}_2\in T_+$ such that:
$$
\mathsf{g}_1\mathsf{g}_2=\bar{h}t^2.
$$
This gives a decomposition of $\mathsf{h}$ into a product of exactly $k$ elements from $T_+$:  
$$
\mathsf{h}=\mathsf{g}_1\mathsf{g}_2t^{k-2}.
$$
We see $||\mathsf{h}||=k$.

In the case $k<-1$, by Lemma \ref{2k} there exist $\mathsf{g}'_1,\mathsf{g}'_2\in T_-$, such that:
$$
\mathsf{g}'_1\mathsf{g}'_2=\bar{h}t^{-2}.
$$
In particular, 
$$
\mathsf{h}=\bar{g}'_1\bar{g}'_2 t^{k+2}
$$
Thus $||\mathsf{h}||=|k|$.
\end{proof}

We now apply Lemma \ref{ng1} to the following lemma.  
\begin{lemma}\label{pm2}
   Let $\mathsf{g}=\bar{g}t^\ell\in G_{\mathbb{Z}}$ and $||\mathsf{g}||=k$.  
Then there exists a geodesic $\overline{S}$-decomposition 
$\mathsf{g} = \mathsf{s}_1,\ldots,\mathsf{s}_k$ such that for each $\mathsf{s}_i$ with non-empty support we have $i_{min}(\mathsf{s}_i)\geq i_{min}(\mathsf{g})-2$ and $i_{max}(\mathsf{s}_i)\leq i_{max}(\mathsf{g})+2$.
\end{lemma}
\begin{proof}
When $\bar{g}t^\ell\in \overline{S}$ or $\bar{g}t^n=g_1g_2$ the statement of the lemma is obvious. 
In cases when $\bar{g}t^\ell=g_1g_2g_3$ such that $\bar{g}$ is not a $[\pm,t]$-commutator or when $\bar{g}t^\ell=\bar{g}t^{\pm1}$ such that $\bar{g}$ is not $[\pm1,t]$-commutator, apply the corresponding places of the proof of Lemma \ref{ng1}.
    
Let us consider the remaining cases. 
To simplify notation let $i_{min}(\mathsf{g})=m_1$ and $i_{max}(\mathsf{g})=m_2$. 

Let $\bar{g}$ be a $[\pm,t]$-commutator, $\omega(g)=3$ and $\ell =0$. 
Then $||\mathsf{g}||=2$. 
According the proof of Lemma \ref{pmk3} there exist $\bar{g}_1,\bar{g}_2\in H$ such that 

$\bar{g}=\alpha(\bar{g}_1)\bar{g}_1^{-1}\bar{g}_2\alpha(\bar{g}_2^{-1})$,  and 

$i_{min}(\alpha(\bar{g}_1)\bar{g}_1^{-1})=i_{min}(\bar{g}_2\alpha(\bar{g}_2^{-1}))=m_1$,  

$i_{max}(\alpha(\bar{g}_1)\bar{g}_1^{-1})=i_{max}(\bar{g}_2\alpha(\bar{g}_2^{-1}))=m_2+1$. \\
Put 
 \begin{align*}
     \mathsf{s}_1&=\alpha(\bar{g}_1)\bar{g}_1^{-1}t^{-1},\\
     \mathsf{s}_2&=\alpha(\bar{g}_2\alpha(\bar{g}_2^{-1}))t. 
 \end{align*}
Then:
 \begin{align*}
     \mathsf{s}_1\mathsf{s}_2&=\alpha(\bar{g}_1)\bar{g}_1^{-1}t^{-1}\alpha(\bar{g}_2\alpha(\bar{g}_2^{-1}))t=\\
     &=\alpha(\bar{g}_1)\bar{g}_1^{-1}\alpha^{-1}(\alpha(\bar{g}_2\alpha(\bar{g}_2^{-1})))=\\
     &=\alpha(\bar{g}_1)\bar{g}_1^{-1}\bar{g}_2\alpha(\bar{g}_2^{-1})=\bar{g}.
 \end{align*}
 It is easy to see that $i_{min}(\mathsf{s}_1)=m_1, i_{max}(\mathsf{s}_1)=m_2+1$, $i_{min}(\mathsf{s}_2)=m_1-1$ and $i_{max}(\mathsf{s}_2)=m_2$ (apply $\alpha$ to $\bar{g}_2\alpha(\bar{g}_2^{-1})$), i.e. the statement of the lemma holds.

Now let $\ell=0$ and $\omega(\bar{g})>3$. According the proof of Lemma \ref{pmk4} and Remark \ref{poz} we can see in a similar way that there exist a geodesic decomposition $\bar{g}=\mathsf{s}_1\mathsf{s}_2$ such that $i_{min}(\mathsf{s}_1)=m_1, i_{max}(\mathsf{s}_1)=m_2+1$, $i_{min}(\mathsf{s}_2)=m_1-1$ and $i_{max}(\mathsf{s}_2)=m_2$.
 
 Now let us consider the situation when $\ell >1$, i.e. $k= ||\mathsf{g}||=\ell$. 
By the proof of Lemma \ref{2k} there exist $\bar{g}_1,\bar{g}_2$ such that 

$\bar{g}=\bar{g}_1\alpha(\bar{g}_1^{-1})\bar{g}_2\alpha(\bar{g}_2^{-1})$, and 

$i_{min}(\bar{g}_1\alpha(\bar{g}_1^{-1}))=i_{min}(\bar{g}_2\alpha(\bar{g}_2^{-1}))=m_1$, 

$i_{max}(\bar{g}_1\alpha(\bar{g}_1^{-1}))=i_{max}(\bar{g}_2\alpha(\bar{g}_2^{-1}))=m_2+1$. \\ 
Then we put 
 \begin{align*}
     \mathsf{s}_1&=\bar{g}_1\alpha(\bar{g}_1^{-1})t,\\
     \mathsf{s}_2&=\alpha^{-1}(\bar{g}_2\alpha(\bar{g}_2^{-1}))t,\\
     \mathsf{s}_3&=\mathsf{s}_4=\ldots=\mathsf{s}_k=t, 
 \end{align*}
and verify: 
 \begin{align*}
     \mathsf{s}_1\mathsf{s}_2\ldots \mathsf{s}_n&=\bar{g}_1\alpha(\bar{g}_1^{-1})t\alpha^{-1}(\bar{g}_2\alpha(\bar{g}_2^{-1}))t\cdot t^{n-2}=\\
     &=\bar{g}_1\alpha(\bar{g}_1^{-1})t\alpha^{-1}(\bar{g}_2\alpha(\bar{g}_2^{-1}))t^{-1}\cdot t\cdot t^{n-1}=\\
     &=\bar{g}_1\alpha(\bar{g}_1^{-1})\alpha(\alpha^{-1}(\bar{g}_2\alpha(\bar{g}_2^{-1})))t^{n}=\\
     &=\bar{g}_1\alpha(\bar{g}_1^{-1})\bar{g}_2\alpha(\bar{g}_2^{-1})t^{n}=\bar{g}t^n.
 \end{align*}
 It is easy to see, that 
 $i_{min}(\mathsf{s}_1)=m_1$, $i_{max}(\mathsf{s}_1)=m_2+1$, 
 $i_{min}(\mathsf{s}_2)=m_1+1$ and $i_{max}(\mathsf{s}_2)=m_2+2$, i.e. the statement of the lemma holds.

Finally consider $\mathsf{g}=\bar{g}t^{-\ell}\in H, \ell>1$. 
Again $k= ||\mathsf{g}||=\ell$. 
According the proof of lemma \ref{2k} there exist $\bar{g}_1,\bar{g}_2$ such that
 
 $\bar{g}=\alpha(\bar{g}_1)\bar{g}_1^{-1}\alpha(\bar{g}_2)\bar{g}_2^{-1}$,
 
 $i_{min}(\alpha(\bar{g}_1)\bar{g}_1^{-1})=i_{min}(\alpha(\bar{g}_2)\bar{g}_2^{-1}=m_1-1$ and 
 
 $i_{max}(\alpha(\bar{g}_1)\bar{g}_1^{-1})=i_{max}(\alpha(\bar{g}_2)\bar{g}_2^{-1})=m_2$.\\ Then we put:
 \begin{align*}
     \mathsf{s}_1&=\alpha(\bar{g}_1)\bar{g}_1^{-1}t^{-1},\\
     \mathsf{s}_2&=\alpha(\alpha(\bar{g}_2)\bar{g}_2^{-1})t^{-1},\\
     \mathsf{s}_3&=\mathsf{s}_4=\ldots=\mathsf{s}_n=t^{-1},
 \end{align*}
 and verify:
 \begin{align*}
     \mathsf{s}_1\mathsf{s}_2\ldots \mathsf{s}_n&=\alpha(\bar{g}_1)\bar{g}_1^{-1}t^{-1}\alpha(\alpha(\bar{g}_2)\bar{g}_2^{-1})t^{-1}\cdot t^{n+2}=\\
     &=\alpha(\bar{g}_1)\bar{g}_1^{-1}t^{-1}\alpha(\alpha(\bar{g}_2)\bar{g}_2^{-1})t\cdot t^{-1}\cdot t^{n+1}=\\
     &=\alpha(\bar{g}_1)\bar{g}_1^{-1}\alpha^{-1}(\alpha(\alpha(\bar{g}_2)\bar{g}_2^{-1}))t^{n}=\\
     &=\alpha(\bar{g}_1)\bar{g}_1^{-1}\alpha(\bar{g}_2)\bar{g}_2^{-1}t^{n}=\bar{g}t^{n}.
 \end{align*}
 It is easy to see, that $i_{min}(\mathsf{s}_1)=m_1-1, i_{max}(\mathsf{s}_1)=m_2, i_{min}(\mathsf{s}_2)=m_1-2$ and $i_{max}(\mathsf{s}_2)=m_2-1$, i.e. the statement of the lemma holds.

Since all possible cases have been verified  we conclude that for any $\mathsf{g}\in G_\mathbb{Z}$ there exist geodesic decomposition of it satisfying the statement of the lemma.
\end{proof}

\subsection{The final part of the proof} 
In this section we view the cyclic group  $\mathbb{Z}_{2n+1}= \mathbb{Z}/(2n+1)\mathbb{Z}$ under its natural presentation on the set 
$\{-n,-n+1,\ldots,-1,0,1,\ldots,n-1,n\}$. 
The groups $P$ and $G_{\mathbb{Z}}$ are taken from Section 5.1. 
Let:
$$
H_{[-n,n]}=\bigoplus_{i\in\mathbb{Z}_{2n+1}}H_i,
$$
where $H_i=P$ for all $i$. 
Let $\alpha_{2n+1}$ be an automorphism of $H_{[-n,n]}$ taking the one-step shift to the left (as before). 
Let:
$$
G_{[-n,n]}=H_{[-n,n]}\rtimes_\eta\mathbb{Z}_{2n+1},
$$
where $\eta: \mathbb{Z}_{2n+1}\to Aut(H_{[-n,n]})$ is defined by the formula: 
$\eta(k)=\alpha_{2n+1}^k$. 
The definitions of $[k,t]$- and $[\pm,t]$-commutators for the group $H_{[-n,n]}$ are preserved as above without any changes.
We consider $G_{[-n,n]}$ as a metric group with respect to the word norm generated by the set $H_0\cup\{t\}$.

\begin{remark}
{\em 
Let the support of $\bar{g} \in H$ be in $[-n+2,n-2]$. 
Then 
$\bar{g}t^{\pm 1}\bar{g}^{-1} \in \overline{S}$ has the form $\bar{h}t^{\pm 1}$ in $G_\mathbb{Z}$, where $\bar{h}$ has support in $[-n+1,n-1]$. 
Note that in $G_{[-n,n]}$ the conjugate $\bar{g}t^{\pm 1}\bar{g}^{-1}$ is of the same form. 
}
\end{remark}

\begin{definition}
We say, that an element 
$\bar{h}=h_{-n}h_{-n+1}\ldots h_n \in H_{[-n,n]}$ has full support if 
$h_{-n}\not= 1 \not =h_n$. 
In the opposite case, we say that the support of $\bar{h}$ is not full.
\end{definition}

\bigskip 

\noindent 
{\em Proof of Theorem \ref{GZmlef}.} 
For any $\bar{h}t^{m} \in G_{[-n,n]}$ define    
$$
N(\bar{h}t^{m}):= \max\{|i_{min}(\bar{h})|, |i_{max}(\bar{h})|, |m|\}.   
$$ 
Let us fix a finite $K\subset G_\mathbb{Z}$ and let 
$N:=\max\{N(\bar{h}t^{m}): \bar{h}t^{m}\in K\}. $

Consider the group $G_{[-2N-3,2N+3]}$. 
Let us define a map 
$\varphi: G_\mathbb{Z}\to G_{[-2N-3,2N+3]}$ as  follows:
$$
\varphi(\bar{h}t^{m})=\left\{
\begin{array}{cl}
\bar{h}t^{m}&\mbox{if } N(\bar{h}t^{m})\leq 2N+2\\
1&\mbox{otherwise.}
\end{array}
\right.
$$
We will show that $\varphi$ is a 
$K$-$\mathbb{Z}$-almost homomorphism. 
It is clear that $\varphi$ is an injection on $K$.
Since for any 
$\bar{h}_1t^{n_1},\bar{h}_2t^{n_2}\in K$ we have 
$N(\bar{h}_1t^{n_1}\cdot\bar{h}_2t^{n_2})\leq 2N$, 
then:
$$
\varphi(\bar{h}_1t^{n_1}\bar{h}_2t^{n_2})=\varphi(\bar{h}_1t^{n_1})\varphi(\bar{h}_2t^{n_2}).
$$ 

Let us prove that for any $\mathsf{g}\in K$ 
\begin{equation}\label{rn}
||\mathsf{g}||_{G_\mathbb{Z}}=||\varphi(\mathsf{g})||_{G_{[-2N-3,2N+3]}}.
\end{equation}
By Lemma \ref{pm2} we know that any $\mathsf{g}\in K$ has a geodesic decomposition in $G_\mathbb{Z}$
$$
\mathsf{g}=\mathsf{s}_1\cdot \ldots \cdot \mathsf{s}_\ell,
$$
such that $\mathsf{s}_i \in \overline{S}$ and $N(\mathsf{s}_i)\leq 2N+2$. 
Since $\varphi(\mathsf{s}_i)=\mathsf{s}_i$, for all $i\leq \ell$,  we see 
$||\mathsf{g}||_{G_\mathbb{Z}}\geq||\varphi(\mathsf{g})||_{G_{[-2N-3,2N+3]}}$.  Note that this inequality is strict only if  
$\mathsf{s_1}\cdot \ldots\cdot\mathsf{s_\ell}$ is not a geodesic decomposition in $G_{[-2N-3,2N+3]}$. Also note that when $\mathsf{g}\in H_{[-2N-3,2N+3]}$ is a product of at most two generators from $S$ then the inequality is, in fact, equality.

Let $\mathsf{g}=\bar{g}t^{k_\mathsf{g}}$ be any element of $G_\mathbb{Z}$. 
By Lemma \ref{ng1} the conjunction of $k_\mathsf{g}=0$ and $\omega(\bar{g})>2$ implies that $\parallel \mathsf{g} \parallel_{G_\mathbb{Z}}\in \{ 2,3\}$ and the value 2 is realized exactly when $\bar{g}$ is a $[\pm,t]$-commutator. 
By the proof of Lemma \ref{ng1} in the latter case the geodesic decomposition of $\mathsf{g}$ is a product of an element from $T_+$ and an element from $T_-$. 
In particular, $2 = ||\mathsf{g}||_{G_\mathbb{Z}}=||\varphi(\mathsf{g})||_{G_{[-2N-3,2N+3]}}$.

If $k_\mathsf{g}>0$, then  
$||\varphi(\mathsf{g})||_{G_{[-2N-3,2N+3]}}\geq k_\mathsf{g}$, because in order to realize $k_\mathsf{g}$-th power of $t$ we need at least $k_\mathsf{g}$ elements from $T_+$. 
By the same reason when $k_\mathsf{g}<0$ 
we need at least $|k_\mathsf{g}|$ elements from 
$T_-$, i.e. 
$||\varphi(\mathsf{g})||_{G_{[-2N-3,2N+3]}}\geq |k_\mathsf{g}|$. 
So we have 
$||\mathsf{g}||_{G_\mathbb{Z}}=||\varphi(\mathsf{g})||_{G_{[-2N-3,2N+3]}}$ for all $\mathsf{g}$, such that $|k_\mathsf{g}|>1$. If $\bar{g}$ is a $[\pm1,t]$-commutator and $|k_\mathsf{g}|=1$ the equality holds by Lemma \ref{ng1}.

These arguments together with Lemma \ref{ng1} lead us to the conclusion that when equality (4) does not hold, one of the following cases must happen:      
\begin{enumerate}
    \item $k_g=\pm1$, $\bar{g}$ is not a $[\pm1,t]$-commutator in $G_\mathbb{Z}$, but 
$\bar{g}$ is a $[\pm1,t]$-commutator in $G_{[-2N-3,2N+3]}$,
    \item $k_g=0$, $\omega(\bar{g})\ge 3$, $\bar{g}$ is not a $[\pm,t]$-commutator in $G_\mathbb{Z}$, but $\bar{g}$ is a $[\pm,t]$-commutator in $G_{[-2N-3,2N+3]}$.
\end{enumerate}
In order to finish the proof of the theorem we will show that these cases are impossible. 
In the first one we put $k_\mathsf{g}=1$ (the case $k_\mathsf{g}=-1$ is similar). 
It is enough to show that if there exists $\bar{h}$ in $H_{[-2N-3,2N+3]}$ with full support such that $\bar{g}=\bar{h}\alpha(\bar{h}^{-1})$, then there exists $\bar{h}'$ with non-full support such that  $\bar{g}=\bar{h}'\alpha(\bar{h}'^{-1})$, i.e. 
$\bar{g}t^{\pm1}$ is a $[\pm1,t]$-commutator in  $G_\mathbb{Z}$. 
 
Denote $i=i_{min}(\bar{g})$ and $j=i_{max}(\bar{g})-i_{min}(\bar{g})$ (i.e. $i+j<2N+3$). Then $\bar{g}$ has the following form:
$$
\bar{g}=(h_{i+j+1}h_{i+1}^{-1})_i(h_{i+1}h_{i+2}^{-1})_{i+1}\ldots(h_{i+j}h_{i+j+1}^{-1})_{i+j} , 
$$
and $\bar{h}$ must be as follows:  
$$
(h_{i+j+1})_{-2N-3}\ldots(h_{i+j+1})_ih_{i+1}h_{i+2}\ldots h_{i+j}h_{i+j+1}(h_{i+j+1})_{i+j+2}\ldots(h_{i+j+1})_{2N+3}, 
$$
Now lets consider 
$\bar{h}'=\bar{h}\bar{h}_{i+j+1}^{-1}$, where 
$\bar{h}_{i+j+1}^{-1}$ is a constant element \\ 
$(h_{i+j+1}^{-1})_{-2N-3}\ldots(h_{i+j+1}^{-1})_{2N+3}$. 
Then we have:
$$
\bar{h}'=(h_{i+1}h_{i+j+1}^{-1})_{i+1}\ldots(h_{i+j}h_{i+j+1}^{-1})_{i+j}.
$$
In particular,  $\bar{h}'$ has non full support and:
$$
\bar{g}=\bar{h}'\alpha(\bar{h}'^{-1}). 
$$

In the second case, without loss of generality we put $\bar{g}=g_1g_2g_3$ (see Lemmas \ref{pos}, \ref{opo} and \ref{pmk4}). 
By Lemma \ref{pmk3} it is enough to show that if $\bar{g}$ is a $[\pm,t]$-commutator in $G_{[-2N-3,2N+3]}$ then $\Xi(g_1,g_2,g_3)$ is satisfied. 
Suppose that there exist $\bar{h}_1,\bar{h}_2\in G_{[-2N-3,2N+3]}$ such that $\bar{g}=\alpha(\bar{h}_1)\bar{h}_1^{-1}\bar{h}_2\alpha(\bar{h}_2^{-1})$. 
Consider the product $\alpha(\bar{h}_1)\bar{h}_1^{-1}\bar{h}_2\alpha(\bar{h}_2^{-1})$:
    \begin{center}
    \begin{tabular}{|c|c|c|c|c|c|c|c|}
\hline
                 Index & $\ldots$ & $0$ & $1$ & $2$ & $3$ & $4$ & $\ldots$ \\ \hline
                $\alpha(\bar{h}_1)$ & $\ldots$ & $c_1$ & $a$ & $b$ & $c_4$ & $c_5$ & $\ldots$ \\ \cline{1-8} 
                 $\bar{h}_1^{-1}$ & $\ldots$ &$c_0^{-1}$ & $c_1^{-1}$ & $a^{-1}$ & $b^{-1}$ & $c_4^{-1}$ & $\ldots$ \\ \hline
                $\bar{h}_2$ & $\ldots$ & $f_0$ & $f_1$ & $d$ & $e$ & $f_4$ & $\ldots$ \\ \cline{1-8} 
                 $\alpha(\bar{h}_2^{-1})$ & $\ldots$ & $f_1^{-1}$ & $d^{-1}$ & $e^{-1}$ & $f_4^{-1}$ & $f_5^{-1}$ & $\ldots$  \\ \hline
                 Result &$\ldots$ &$1$ & $g_1$ & $g_2$ & $g_3$ & $1$ & $\ldots$ \\ \hline
\end{tabular}
\end{center}
The equation $c_5c_4^{-1}f_4f_5^{-1}=1$ appears  at position 4. 
It is equivalent to $c_4^{-1}f_4=c_5^{-1}f_5$. Similarly we have: $c_5^{-1}f_5=c_6^{-1}f_6$ at position 6. 
Moving to the right till position 0 (making a loop) we wright down the corresponding equations and deduce from them the conditions $c_i^{-1}f_i=c_4^{-1}f_4$ for all $i=-2N-3,-2N-2,\ldots,0,1,4,5,\ldots,2N+3$. 
Positions 1,2,3 give the system of equations:
$$
\left\{
\begin{array}{l}
g_1=ac_4^{-1}f_4d^{-1}\\
g_2=ba^{-1}de^{-1}\\
g_3=c_4b^{-1}ef_4^{-1}.
\end{array}
\right.
$$
From the first one we have $d=g_1^{-1}ac_4^{-1}f_4$. 
Substituting it to second equation we have:
$$
\left\{
\begin{array}{l}
d=g_1^{-1}ac_4^{-1}f_4\\
g_2=ba^{-1}h_1^{-1}ac_4^{-1}f_4e^{-1}\\
g_3=c_4b^{-1}ef_4^{-1}. 
\end{array}
\right.
$$
Thus $e=g_2^{-1}ba^{-1}g_1^{-1}ac_4^{-1}f_4$ (see the second equation). 
Substituting this expression of $e$ to the third equation we have:
$$
\left\{
\begin{array}{l}
d=g_1^{-1}ac_4^{-1}f_4\\
e=g_2^{-1}ba^{-1}g_1^{-1}ac_4^{-1}f_4\\
g_3=c_4b^{-1}g_2^{-1}ba^{-1}g_1^{-1}ac_4^{-1},  
\end{array}
\right.
$$
i.e. $\Xi(g_1,g_2,g_3)$ holds.

We conclude that equation \eqref{rn} holds for any $\mathsf{g}\in K$. 
As a result $\varphi$ is a metric $K$-$\mathbb{Z}$-almost homomorphism. 
The group $G_\mathbb{Z}$ is metrically LEF.
$\Box$

\section{Length functions and logic} 

Let $G$ be a group.
Symbols $<,=,>$ will be interpreted below in the standard way.  
Since the condition $d_{\ell}(g,h) \Box \varepsilon$ (with $\varepsilon \in \mathbb{Q}$ and $\Box \in \{ <,>,= \}$) is equivalent to $\ell(gh^{-1}) \Box \varepsilon$ it is easier to view the space of metrics on $G$ as a space of length functions. 
On the other hand a pseudo length function $\ell :G \rightarrow \Lambda$ can be viewed as a function 
$\mathsf{f}_{\ell}: G\times (\mathbb{Q} \cap \Lambda )\rightarrow \{ < , >, = \}$ defined by 
\[ 
\mathsf{f}_{\ell} (g,r) = \Box \, \Leftrightarrow  \, \ell(g) \, \Box \, r   \mbox{, where } \, \Box \in \{ <,>,= \}. 
\]
This gives the correspondence  
\[ 
\mbox{(pseudo) metric }d_{\ell}  \, \Leftrightarrow \, \mbox{(pseudo) norm } \ell \,  \Leftrightarrow  \,  \mathsf{f}_{\ell}. 
\]
The set of invariant pseudo metrics $G \times G\rightarrow \Lambda$ becomes a subspace of the compact space  
$ \{ < , >, = \}^{G\times (\mathbb{Q} \cap \Lambda )}$ (under the product topology). 
We will see in Section 3.1 that when $G$ is countable this subspace is $G_{\delta}$. 

The set $\mathcal{N}(G)$ of all normal subgroups of $G$ can be viewed as a closed subspace of the space $2^G$ under the product topology, see Section 3.4 in \cite{CSC}. 
In order to formulate a metric version of the approach of Section 7.1 of \cite{CSC} we will consider some closed subspace of the product  
$\mathcal{N}(G)\times \{ < , >, = \}^{G\times (\mathbb{Q} \cap \Lambda )}$.  
Sections 5.1 and 5.2 below are devoted to this task. 
In Section 5.3 we consider the case of word norms. 
In Section 5.4 we view metrics as families of binary relations. 
Then metric groups become firts-order structures. 
We will see some (may be expected) connections between logic properties of these structures and local embeddability. 
In Section 5.5 we give a characterization of metric local embeddability in terms of direct limits.

\subsection{The space of pseudo length functions} 

In this section we preserve the set up and the notation from Introduction. 
In order to show that the subspace of invariant pseudo-norms of a countable group $G$ is $G_{\delta}$ let us introduce the following definition. 
(In this definition we do not assume that the group is countable.)

\begin{definition} \label{weight} 
We will say that a function $f: (G\times \Lambda ) \rightarrow \{ <,=,>\}$ is a {\bf weight function} on $G$ if the following conditions hold: \\  
(1) if $q \le q'$ and $f(g,q) \in \{ <,=\}$, then $f(g,q') \in \{ <,=\}$, \\ 
(2) if $q \le q'$ and $f(g,q') \in \{ >,=\}$, then $f(g,q) \in \{ >,=\}$, \\ 
(3) $f(1,0) \in \{ = \}$ and $f(g,0) \in \{ =, > \}$ for all $g\in G$, \\ 
(4) $f(g,q) = f(g^{-1},q)$ for all $g\in G$. \\
Let $W_{\Lambda}(G)$ denote the set of weight functions on $G$. 
\end{definition}
We say that $f\in W_{\Lambda}(G)$ 
is {\em invariant} if $f(h^{-1}gh,q) = f(g,q)$ for all $g,h\in G$ and $q\in \Lambda \cap \mathbb{Q}$. 

Each $f\in W_{\Lambda}(G)$ defines the function $\mathsf{w}_f : G\rightarrow \Lambda \cup \{\infty\}$ with the following definition: 
\[ 
\mathsf{w}_f (g) = r \Leftrightarrow r = \mathsf{inf} \{ q \in \Lambda \cap \mathbb{Q} : f(g,q) \in \{ =,<\} \} . 
\] 
(when the set from the right side is empty $\mathsf{w}_f$ takes value $\infty$). 

Note that under the ordering of the set $\{ < , = , >\}$ that $<$ is the minimum and $>$ is the maximum, we obviously have that 
\[ 
f_1 \le f_2 \Leftrightarrow \mathsf{w}_{f_1} \le \mathsf{w}_{f_2} . 
\]
Furthermore, $f$ is reconstructed from $\mathsf{w}_f$ as follows. 
\[ 
f(g,q) = \Box \Leftrightarrow \mathsf{w}_f (g) \Box q. 
\] 
The following lemma follows from these considerations. 

\begin{lemma} 
Let $\ell : G \rightarrow \Lambda$ be a pseudo-norm and let $f \in W_{\Lambda}(G)$ coincide with $\mathsf{f}_{\ell}$. 
Then $\mathsf{w}_{f} = \ell$. 
\end{lemma}

\begin{itemize}   
\item Let $IW_{\Lambda}(G)$ consist of invariant weight functions. 
\item Let $L_{\Lambda}(G)$ (resp. $PsL_{\Lambda}(G)$) be the subset of all $f \in W_{\Lambda}(G)$ for which  
$\mathsf{w}_f$ is a $\Lambda$-norm (resp. is a pseudo-$\Lambda$-norm). 
\item  Let $IL_{\Lambda}(G)$ and $IPsL_{\Lambda}(G)$ be the corresponding subsets  of invariant functions. 
\item Let $\bar{\Lambda} = \Lambda \cup \{ \infty\}$. 
We will assume that 
$\infty + \infty = \infty$.  
\item  Let $L_{\bar{\Lambda}}(G)$ (resp. $PsL_{\bar{\Lambda}}(G)$) be the subset of $f \in W_{\Lambda}(G)$ where 
$\mathsf{w}_f$ satisfies conditions (i') - (iii)  (resp. (i) - (iii)) of Definition \ref{lf} (and the value $\infty$ is possible). 
\item  Let $IL_{\bar{\Lambda}}(G)$ and $IPsL_{\bar{\Lambda}}(G)$ be the corresponding subsets  of invariant functions. 
\end{itemize}

\begin{lemma} \label{closed} 
(1) The sets $W_{\Lambda}(G)$, $\, IW_{\Lambda}(G)$, 
$\, PsL_{\bar{\Lambda}}(G)$, 
and 
$\, IPsL_{\bar{\Lambda}}(G)$ are closed in $\{ <,=,> \}^{G\times (\Lambda \cap \mathbb{Q})}$.  

(2) Let $G$ be countable. 
Then the sets \, $PsL_{\Lambda} (G)$,  $L_{\Lambda}(G)$,  $IPsL_{\Lambda}(G)$  and  
$\, IL_{\Lambda}(G)$ belong to 
$G_{\delta}$ in 
$\{ <,=,> \}^{G\times (\Lambda \cap \mathbb{Q})}$. 
\end{lemma}

\begin{proof} 
Since for every $g\in G$, $q\in \mathbb{Q}$ and $\Box \in \{ <,=,>\}$ the set 
$$
\{ f\in  \{ < , >, = \}^{G\times (\mathbb{Q} \cap \Lambda )} \, | \, f(g,q) = \Box \}
$$ 
is clopen, we easily see that both $W_{\Lambda}(G)$ and $IW_{\Lambda}(G)$ are closed in \\ 
$\{ < , >, = \}^{G\times (\mathbb{Q} \cap \Lambda )}$. 

When $f\in W_{\Lambda}(G)$ and $\ell = \mathsf{w}_{f}$ as above then the triangle condition for $\ell$ is equivalent to the conjunction of all conditions of the following form:  
\begin{quote} 
if $f(g,q) \in \{ <,=\}$ and $f(g',q') \in \{ <,=\}$, then $f(gg',q + q') \in \{ <,=\}$.  
\end{quote} 
Each of them defines a closed set, so the triangle
condition corresponds to a closed set too. 
Now statement (1) is clear. 

To see statement (2) note that $L_{\Lambda}(G)$ is contained in the intersection of all sets 
\[ 
P_{g} = \{ f\in  \{ < , >, = \}^{G\times (\mathbb{Q} \cap \Lambda )} \, | \, \exists q (f(g,q) \in \{ > \} \wedge q>0 ) \wedge \, \exists q' (f(g,q') \in \{ =, <\} \} , 
\] 
where $g\not = 1$.  
These sets are open. 
The rest of this case is easy. 
Similar arguments work for remaining cases.  
\end{proof}

When $\mathcal{D}$ is a closed subset of 
$\{ <,=,> \}^{G\times (\Lambda \cap \mathbb{Q})}$  
the statements of Lemma \ref{closed} hold for the  corresponding intersections with $\mathcal{D}$ too. 
For example let us fix some $f_0 \in W_{\Lambda}(G)$. 
Then the set $\{ f \in W_{\Lambda}(G) \, | \, f \le f_0 \}$ is closed in  $\{ <,=,> \}^{G\times (\Lambda \cap \mathbb{Q})}$. 
Let \\ 
$\, W_{\Lambda}(G,f_0)$, $\, IW_{\Lambda}(G,f_0)$, $\, L_{\Lambda}(G,f_0)$, $\, PsL_{\Lambda}(G,f_0)$, $\, IL_{\Lambda}(G,f_0)$, \, $IPsL_{\Lambda}(G,f_0 )$ 
be the intersections of the corresponding \\  
$W_{\Lambda}(G)$, \, \, \, $IW_{\Lambda}(G)$, \, \, \, $L_{\Lambda}(G)$, \, \, \, $PsL_{\Lambda}(G)$, \, \, \, $IL_{\Lambda}(G)$ \, \,  or \, \, $IPsL_{\Lambda}(G)$ \\ 
with $\{ f \in W_{\Lambda}(G) \, | \, \mathsf{w}_f \le \mathsf{w}_{f_0} \}$. \\ 
Then {\em the statement of Lemma \ref{closed} holds for these sets too. } 
If $\mathcal{D}$ is a $G_{\delta}$-subset then the corresponding intersections are $G_{\delta}$ too. 
We now give a more interesting example. 

\begin{example} 
The subset of $\, PsL_{\Lambda}(G)$ which corresponds to integer-valued pseudo-norms is $G_{\delta}$. 
{\em Indeed, having $g\in G$ and 
$q_1 ,q_2 \in \Lambda \cap \mathbb{Q}$ with $q_1 < q_2$ and $[ q_1,q_2 ]\cap \mathbb{N} = \emptyset$ the set 
\[ 
\{ f\in PsL_{\Lambda}(G) \, | \, f(g,q_1 ) \in \{ <\} \vee f(g, q_2) \in \{ > \} \}
\] 
is clopen. 
The rest is easy. } 
\end{example}

\subsection{Parametrization of quotients} 

In this subsection we assume that $G$ is countable. 
Let 
\[ 
\mathcal{NM}(G) = \{ (N,f)\in \mathcal{N}(G) \times IPsL_{\Lambda}(G) \, | \, N\le \mathsf{ker}  (\mathsf{w}_f)\} .
\] 
This set parametrizes all quotients $G/N$ considered as pseudo-normed groups with respect to $\mathsf{w}_f$. 

\begin{lemma} 
The set $\mathcal{NM}(G)$ is a $G_{\delta}$-subset of $\mathcal{N}(G) \times IPsL_{\Lambda}(G)$. 
\end{lemma} 

\begin{proof}  
The condition $N\le \mathsf{ker}  (\mathsf{w}_f)$ says that for any finite $K\subset G$ we have: 
\[ 
\forall g\in K \, ( g\not\in N \, \vee \, f(g,0) \in \{ = \} ). 
\] 
This condition defines a clopen set of pairs $(N,f)$. 
\end{proof} 

We now see that the space $\mathcal{NM}(G)$ is a Polish space. 
The following theorem is a metric version of Theorem 7.1.16 from \cite{CSC}. 

\begin{theorem} \label{dir_set}
Let us assume that: 
\begin{itemize} 
\item $\mathcal{C}$ is a class of invariant pseudo-normed groups; 
\item $G$ is a countable group, $(N,f) \in \mathcal{NM}(G)$ and $\ell$ is the corresponding pseudo-norm $\mathsf{w}_f$ (i.e $f= \mathsf{f}_{\ell}$);  
\item $(I, \prec )$ is a directed set and 
 $\{ (N_i ,f_i ) \, | \, i\in I\}$ is a net from $\mathcal{NM}(G)$ which converges to $(N,f)$ and, moreover, for every $i\in I$ the group $(G/N_i , \ell_i )$ with $\ell_i = \mathsf{w}_{f_i}$, belongs to the class $\mathcal{C}$. 
\end{itemize} 
Then $(G/N, \ell)$ is LE$\mathcal{C}$. 
\end{theorem}

\begin{proof} 
Let us denote by 
$$
\rho: G \to G/N 
\mbox{ and }  
\rho_i: G\to G/N_{i} \mbox{, where } i\in I, 
$$ 
the corresponding quotient homomorphisms. 
Since $N\le \mathsf{ker}  (\ell )$ the pseudo-norm $\ell$ induces a pseudo-norm on $G/N$. 
We denote it by $\ell$ too. 
Similarly each group $G/N_i$ can be vewed as a pseudo-normed group with respect to $\ell_i$. 

Let $K$ be a symmetric finite subset of $G$ with $1\in K$ and let $Q$ be a finite subset of $\Lambda \cap \mathbb{Q}$ with $0\in Q$. 
Find $i_0 \in I$ such that for all $i\ge i_0$ 
\begin{equation}\label{F4zgodnosc}
N\cap K^4 = N_i \cap K^4 \mbox{ and } f\upharpoonright_{K^4\times Q} \, \, = \, f_{i}\upharpoonright_{K^4\times Q}. 
\end{equation} 
Having $C=(G/N_{i_0},\ell_{i_0})\in \mathcal{C}$ define a map $\varphi: G/N \to C$ as follows:
\begin{equation}\label{ahom}
    \varphi(g)=\left\{
\begin{array}{cl}
\rho_{i_0}(a)&\mbox{if }g=\rho(a)\mbox{ for some }a\in K^2\\
1&\mbox{ otherwise.}
\end{array}
\right.
\end{equation}
The straightforward arguments show that $\varphi$ is well defined and the pair $(\varphi, \rho(K))$ satisfies the first part of the condition from Definition \ref{mLEFm} (i.e. in terms of \cite{CSC} it is a $\rho(K)$-almost-homomorphism into $G/N_{i_0}$, see Theorem 7.1.16 in \cite{CSC} for a detailed verification). 
In particular, $\varphi$ is injective on $\rho (K)$. 

In order to verify the second part of Definition \ref{mLEFm} take any $q\in Q$ and $a\in K$. 
Then applying conditions $N\le \mathsf{ker}(\ell )$ and $N_i\le \mathsf{ker} (\ell_i )$ we see that $\ell(a)=\ell (\rho(a))$ and $\ell_{i}(a)=\ell_{i}(\rho_{i}(a))$ for $i\ge i_0$. 
Thus $\mathsf{f}_{\ell }(a,q)=\mathsf{f}_{\ell}(\rho(a),q)$
and $\mathsf{f}_{\ell_{i}}(a,q) = \mathsf{f}_{\ell_{i}}(\rho_{i}(a), q)$ for $i\ge i_0$. 
Applying \eqref{F4zgodnosc} we obtain $\mathsf{f}_{\ell}(\rho (a),q)=\mathsf{f}_{\ell_i}(\rho_{i}(a),q)$ for $i\ge i_0$, i.e.  $\mathsf{f}_{\ell}(\rho(a),q)= \mathsf{f}_{\ell_{i_0}}(\varphi(\rho (a)),q)$. 
As a result  
$$ 
\ell(\rho(a)) \Box q \Leftrightarrow  \ell_{i_0}(\varphi(\rho (a))) \Box q \mbox{ for all } \Box \in \{ <,=,> \}.
$$
We now see that $\varphi$ is a $\rho(K)$-$Q$-almost-homomorphism into $C$. 
Applying Lemma \ref{alm_h} we finish the proof of the theorem. 
\end{proof} 

\bigskip 

We now give a kind of converse to Theorem \ref{dir_set}. 
It will be applied to statements that in some situations LE$\mathcal{C}$ becomes fully residually $\mathcal{C}$. 

Let 
$\mathsf{n} \in \mathbb{N} \cup \{ \omega \}$. 
Consider the free group $\mathsf{F_n}$ as $G$ in Theorem \ref{dir_set}. 
From now on in this subsection $\mathcal{C}$ will be a class of invariant pseudo-$\Lambda$-normed groups which is closed with respect to taking subgroups.

\begin{theorem} \label{net2} 
Let a pseudo-norm $\ell$ on 
$\mathsf{F_n}$ correspond to some 
$f\in IPsL_{\Lambda}(\mathsf{F_n})$ as 
$f = \mathsf{f}_{\ell}$. 
Assume that $(N,f)\in \mathcal{NM}(\mathsf{F_n})$ 
and the pseudo-normed group 
$(\mathsf{F_n}/N, \ell )$ is LE$\mathcal{C}$. 

Then there is a net $\{ (N_i,f_i )\, | \, i\in I\}$ in  $\mathcal{MN}(\mathsf{F_n})$ which converges to $(N,f)$ and for every $i\in I$ the group $(\mathsf{F_n} /N_i, \ell_i )$ with $\ell_i = \mathsf{w}_{f_i}$, belongs to the class $\mathcal{C}$. 
\end{theorem} 

\begin{proof} 
We will modify the proof of Theorem 7.1.19 of \cite{CSC}. 
Let $X \subset \mathsf{F_n}$ be a free base of $\mathsf{F_n}$ and let $\rho :\mathsf{F_n} \to \mathsf{F_n}/N$ be the quotient homomorphism. 
Let  
\[ 
I = \{ (E,Q) |  \, E \subset \mathsf{F_n} \, , E \, \mbox{ is finite }  \, , \, Q\subset  \Lambda \cap \mathbb{Q} \, , \, Q \mbox{ is finite with } 0\in Q \} 
\] 
partially ordered by inclusions. 
Given $(E,Q) \in I$, we denote by $U_E \subset X$ the subset of $X$ consisting of all elements which appear in the reduced form of
some element in $E$. 
Since $E$ is finite, there exists $n_E \in \mathbb{N}$ such that each $w \in E$ may be written as a word of length $\le n_E$ in $U_E\cup U^{-1}_E$. 
Set
$K_E =\rho(B_{n_E}) \subset \mathsf{F_n}/N$, where $B_{n_E}$ denotes the ball of radius $n_E$ centered at the identity element $1_F$ in the Cayley graph of the subgroup of $\mathsf{F_n}$ generated by $U_E$. 
Since $(\mathsf{F_n}/N,\ell )$ is locally embeddable into $\mathcal{C}$, for every finite $E$ and $Q\subset \Lambda \cap \mathbb{Q}$ with $0\in Q$ there exists 
$(C_{E,Q} ,\ell_{E,Q})\in \mathcal{C}$ and an $K_E$-$Q$-almost-homomorphism 
$\varphi_{E,Q} :(\mathsf{F_n}/N,\ell ) \to C_E$. 
We choose $\varphi_{E,Q}$ so that it is a $n_E$-$K_E$-almost-homomorphism in terms of \cite{CSC}, page 240. 
As $\mathsf{F_n}$ is a free group with base $X$, there exists a unique homomorphism 
$\rho_{E,Q} :\mathsf{F_n} \to C_{E,Q}$ such that 
$\rho_{E,Q} (x) = \varphi_{E,Q}(\rho(x))$ for all $x\in X \cup B_{n_E}$. 
Setting $N_{E,Q} = \mathsf{ker}(\rho_{E,Q})$ we have 
$N_{E,Q} \in \mathcal{N}(F)$. 
Moreover, since $\mathcal{C}$ is closed under
taking subgroups, we have that the quotient group 
$(\mathsf{F_n}/N_{E,Q},\ell_{E,Q})$ also belongs to $\mathcal{C}$.

Consider the pseudo length function on $\mathsf{F_n}$ which is the preimage of $\ell_{E,Q}$, i.e. for any $w\in \mathsf{F_n}$ it takes the value $\ell_{E,Q}(wN_{E,Q})$. 
We denote it by $\ell_{E,Q}$ too. 
Then $N_{E,Q} \le \mathsf{ker}(\ell_{E,Q})$. 
Let $f_{E,Q}$ be $\mathsf{f}_{\ell_{E,Q}}$. 
In this way we have obtained $(N_{E,Q},f_{E,Q})\in \mathcal{NM}(\mathsf{F_n})$. 

We claim that the net 
$\{ (N_{E,Q},f_{E,Q}) \, | \, (E,Q)\in I\}$ converges to $(N,\mathsf{f}_{\ell})$. 
Fix finite $E_0 \subset \mathsf{F_n}$ and finite $Q_0 \subset \Lambda \cap \mathbb{Q}$. 
Let $(E,Q) \in I$ be such that $E_0 \subset E$ (i.e. $\rho (E_0 )\subseteq K_E$) and $Q_0 \subset Q$. 
The argument of the proof of Theorem 7.1.19 from \cite{CSC}, p. 241, shows that $N \cap E_0 = N_{E,Q} \cap E_0$. 
On the other hand $\varphi_{E,Q}$ is chosen to be a $K_E$-$Q$-almost-homomorphism 
$(\mathsf{F_n}/N,\ell ) \to (\mathsf{F_n}/N_{E,Q}, \ell_{E,Q})$. 
In particular $f$ and $f_{E,Q}$ agree for all pairs $(g,q) \in E_0 \times Q_0$. 
\end{proof} 

\bigskip 

The following corollary is a version of Corollary 7.1.21 from \cite{CSC}. 
We mention that in Corollary \ref{r_f_2} 
this statement will be further developed.

\begin{corollary} \label{r_f}
Assume that $G$ is a finitely presented
group and $\ell$ is an invariant pseudo-norm on $G$. 
If $(G, \ell )$ is locally embeddable into $\mathcal{C}$, then $(G,\ell )$ is fully residually $\mathcal{C}$. 
\end{corollary} 

\begin{proof} 
Let $(G,\ell )$ be as in the formulation and 
$\rho : \mathsf{F}_n \rightarrow \mathsf{F}_n /N$ be a quotient homomorphism with $G = \mathsf{F}_n /N$, $n\in \mathbb{N}$.
In order to show that $(G,\ell )$ is fully residually $\mathcal{C}$ let us  
fix a finite $Q \subset \Lambda \cap \mathbb{Q}$ and a finite $K \subset \mathsf{F}_n$ which represents pairwise distinct elements in $\mathsf{F}_n/N$. 
Since $G$ is finitely presented, the subgroup $N$ is normally generated by some finite subset $R \subset \mathsf{F}_n$. 
By Theorem \ref{net2}, there exists a net
$\{ (N_i , f_i ) \, | \, i \in I \} \subset  \mathcal{NM}(\mathsf{F}_n)$ converging to $(N, \mathsf{f}_{\ell})$ such that for every $i \in I$ the pseudo metric group $(\mathsf{F}_n/N_i ,\ell_i )$ belongs to $\mathcal{C}$, where $\ell_i = \mathsf{w}_{f_i}$. 
We can now find $i_0 \in I$ such that
$N \cap (R \cup K) = N_{i_0} \cap (R \cup K)$ and 
$f_{i_0}$ agrees with $\mathsf{f}_{\ell}$ on pairs $K\times Q$. 
As $R \subset N$, this implies $R \subset N_{i_0}$ and hence $N \le N_{i_0}$. 
Thus, there is a canonical epimorphism 
$\varphi : G = \mathsf{F}_n/N \to \mathsf{F}_n/N_{i_0}$ such that $\ell$ agrees with $\ell_{i_0}$ on $\varphi (K)\times Q$. 
\end{proof} 

\begin{remark} \label{quasi} 
{\em In Theorem \ref{net2} and Corollary \ref{r_f} we can make the situation slightly wider. 
As in Proposition \ref{f_g_1} 
assume that $w(z_1,z_2,\ldots, z_m )$ is a quasi-linear group word and $V_w$ be the variety of groups satisfying the corresponding  identity. 
As before by $\mathsf{F}^w_n$ we denote the free $n$-generated group of $V_w$. 
We additionally assume that $n \not= \omega$. 
Then the statement of Theorem \ref{net2} still holds if the free group $\mathsf{F}_n$ is replaced by $\mathsf{F}^w_n$. 
The proof should be modified as follows. 
Put $U_E = X$ and take $n_E$ to be greater than the length of $w(\bar{z})$. 
Then the group $C_{E,Q}$ can be taken from $V_w$. 

If in Corollary \ref{r_f} we assume that $G$ is finitely presented in $V_w$ then the statement still holds. 
To see this replace in the corresponding proof 
$\mathsf{F}_n$ by $\mathsf{F}^w_n$ and apply Theorem \ref{net2} in the modified form as above.  
} 
\end{remark}

\subsection{Approximations of metrically LEF groups with word norms} 

In this subsection from now on we fix a class $\mathcal{C}$ of invariant pseudo-$\Lambda$-normed groups which is closed with respect to taking subgroups.  

Given a group $(G,\ell_G)$ with an invariant pseudo-norm let 
\[ 
\mathcal{NM}(G,\ell_G) = \{ (N,f) \in \mathcal{NM}(G) \, | \, \ell_f \le \ell_G \} .  
\] 
When \, $(N,f)\in \mathcal{NM}(G,\ell_G)$, \, then \, $\ell_f$ \, is viewed as a pseudo-norm on \, $G/N$, \, and \, $(G/N, \ell_f)$ \, is viewed as a homomorphic image of $(G,\ell_G)$ with respect to a {\em metric} homomorphism, see Definition \ref{h}. 
It is easy to see that the subspace 
$\mathcal{NM}(G,\ell_G)$ is closed in  $\mathcal{NM}(G)$. 
We will show in this subsection that it naturally arises when 
$(G,\ell_G) = (\mathsf{F_n} , \parallel \cdot \parallel )$, where the word norm $\parallel \cdot \parallel$ is associated with the standard free basis of $\mathsf{F_n}$. 
The following theorem is a version of Theorem \ref{net2} in the case of word metrics. 
We will see below that in fact it describes the situation when LE$\mathcal{C}$ becomes metrically fully residually $\mathcal{C}$.

\begin{theorem} \label{net3}  
Let an invariant pseudo-norm $\ell$ on $\mathsf{F_n}$ correspond to some 
$f\in IPsL_{\Lambda}(\mathsf{F_n})$ as $\ell = \mathsf{w}_f$ and for each free generator $x_i$, $\ell(x_i) \le 1$ (i.e. $\parallel x_i \parallel$). 
Assume that $(N,f)\in \mathcal{NM}(\mathsf{F_n})$ and the pseudo-normed group 
$(\mathsf{F_n}/N, \ell )$ is LE$\mathcal{C}$. 

Then there is a net $\{ (N_i,f_i )\, | \, i\in I\}$ in $\mathcal{MN}(\mathsf{F_n}, \parallel \cdot \parallel )$ which converges to $(N,f)$ and for every $i\in I$ the group $(\mathsf{F_n} /N_i, \ell_i )$ with $\ell_i = \mathsf{w}_{f_i}$, belongs to the class $\mathcal{C}$. 
\end{theorem} 

\begin{proof} 
We just show how the proof of Theorem \ref{net2} should be modified. 
We preserve the notation of that proof: 
\begin{quote}  
$\rho :\mathsf{F_n} \to \mathsf{F_n}/N$, 
the directed set $I$,  and
$U_E$, $n_E$, $K_E =\rho(B_{n_E})$  for $(E,Q) \in I$. 
\end{quote}
Given finite $E$ and $Q\subset \Lambda \cap \mathbb{Q}$ with $0,1\in Q$ find 
$(C_{E,Q} ,\ell_{E,Q})\in \mathcal{C}$ and an $K_E$-$Q$-almost-homomorphism 
$\varphi_{E,Q} :(\mathsf{F_n}/N,\ell ) \to C_{E,Q}$ as before. 
Let us assume that $U_E = \{ x_0, \ldots , x_k\}$, i.e. $B_{n_E}$ is a finite subset of $\mathsf{F}_{k+1}$. 
Define 
$\pi_{E,Q}: (\mathsf{F_n},\parallel \cdot \parallel )\to (C_{E,Q},\ell_{E,Q})$ according the following values for generators:
\[  
\pi_{E,Q}(x_i)=\left\{
\begin{array}{cl}
\varphi_{E,Q}(\rho(x_i))&\mbox{ if } i\leq k\\
1&\mbox{ otherwise.}
\end{array}
\right.
\] 
By the choice of $\varphi_{E,Q}$ we have 
$\pi_{E,Q} (w(\bar{x})) = \varphi_{E,Q} (\rho (w(\bar{x})))$ for all 
$w(\bar{x}) \in B_{n_E}$. 
Furthermore,  
$N_{E,Q} = \mathsf{ker} (\pi_{E,Q}) \in \mathcal{N}(F)$ and 
$(\mathsf{F_n}/N_{E,Q},\ell_{E,Q})\in \mathcal{C}$. 
Since $1\in Q$ and for each $i$, $\parallel  x_i \parallel =1$,   
\[ 
\ell_{E,Q}(\pi_{E,Q}(x_i))= \ell_{E,Q} (\varphi_{E,Q}(\rho(x_i)))=\ell(\rho(x_i))\leq 1. 
\]  
Since $\ell_{E,Q}$ is conjugacy invariant, we have that for any $v(\bar{x})\in \mathsf{F_n}$    
\[ 
\ell_{E,Q}(\pi_{E,Q}( v(\bar{x})^{-1} x_i v(\bar{x}))) = \ell_{E,Q} (\pi_{E,Q} (x_i )) \le 1. 
\] 
We view $\ell_{E,Q}$ as a pseudo-norm on $\mathsf{F_n}$, i.e. identifying 
$\ell_{E,Q} (w(\bar{x}))$ with 
$\ell_{E,Q} (\pi_{E,Q} (w(\bar{x})))$. 

Note that $\ell_{E,Q}\leq \parallel \cdot \parallel$. 
Indeed, if $g\in \mathsf{F_n}$ with $\parallel g\parallel =k$, is presented as  $g= s_1\ldots s_k$ by $\pm 1$-powers of generators and/or their conjugates, then 
\[ 
\ell_{E,Q}(g)=\ell_{E,Q}(s_1\ldots s_k)\leq\sum_{i=1}^k\ell_{E,Q}(s_i)\leq k = \parallel g \parallel. 
\footnote{this trick is from Lemma 2.7 w \cite{Karl}}
\] 
As a result we see that if $f_{E,Q}$ is defined as in the proof of Theorem \ref{net2} then $(N_{E,Q},f_{E,Q})\in \mathcal{NM}(\mathsf{F_n}, \parallel \cdot \parallel )$. 
The rest repeats the final part of that proof of Theorem \ref{net2}. 
\end{proof} 

The following corollary follows from the proof of Theorem \ref{net3}. 
It is a metric version of Corollary 7.1.21 from \cite{CSC} and a version of Corollary \ref{r_f}. 

\begin{corollary} \label{r_f_2} 
Assume that 
\begin{itemize} 
\item $G$ is a finitely presented group of the form $G = \mathsf{F}_n/N$, $n\in \mathbb{N}$,  
\item $\ell$ is a word norm defined on $G$ with respect to the generators $x_i N$, $1 \le i \le n$, where 
$\{ x_1 ,\ldots , x_n\}$ is a free basis 
of $\mathsf{F}_n$,  
\item the normed group 
$(G, \ell )$ is LE$\mathcal{C}$. 
\end{itemize} 

Then $(G,\ell )$ 
is metrically fully residually $\mathcal{C}$.
\end{corollary} 

\begin{proof} 
Since $G$ is finitely presented, the subgroup $N$ is normally generated by some finite subset $R \subset \mathsf{F}_n$. 
In order to show that $(G, \ell )$ is metrically fully residually LE$\mathcal{C}$ fix a finite $Q \subset \Lambda \cap \mathbb{Q}$ with $0,1 \in Q$ and a finite $K \subset \mathsf{F}_n/N$ represented by some finite $E\subset \mathsf{F}_n$.  
Take $E'\subset \mathsf{F}_n$ such that $E'$ contains $R$, all $x_i$, $1 \le i \le n$, and is of the form $B_{n_E}$ as in the proof of Theorem \ref{net3}.  
We may assume that $K=\rho (E')$ and in fact $E=E'$.  
Find 
$(C ,\ell_{C})\in \mathcal{C}$ and an $K$-$Q$-almost-homomorphism 
$\varphi_{C} :(\mathsf{F_n}/N,\ell ) \to (C,\ell_C )$ as before. 
The restriction of $\varphi_C$ to $\{ x_1 , \ldots , x_n \}$ defines a homomorphism 
from $\mathsf{F_n}$ onto $C$. 
Let $N_C$ be its kernel. 
Since $N \cap E = N_{C} \cap E$ and 
$R \subset N$, we have $R \subset N_{C}$ and hence $N \le N_{C}$. 
Thus, there is a canonical epimorphism 
$\varphi : G = \mathsf{F}_n/N \to \mathsf{F}_n/N_{C}$ which agrees with $\varphi_C$ on $K$. 
Then $\ell$ agrees with $\ell_{C}$ for all pairs $(\varphi (wN), q)$ with $w \in K$ and $q\in Q$.

Since $1\in Q$ and for each $i\le n$, $\ell (x_i N) =1$,   
\[ 
\ell_{C}(\varphi_{C}(x_i N))= \ell (\varphi (x_i N))=1. 
\]  
Since $\ell_{C}$ is conjugacy invariant, we have that for any $v(\bar{x})\in \mathsf{F}_n$  
\[ 
\ell_{C}(\varphi ( v(\bar{x})^{-1} x_i v(\bar{x})N)) = \ell_{C} (\varphi (x_i N)) = 1.
\] 
Note that $\ell_{C}\leq \ell$. 
Indeed, if $g\in \mathsf{F_n}/N$ with $\ell( g) =k$, is presented as  $g= s_1\ldots s_k$ by $\pm 1$-powers of generators and/or their conjugators, then
\[ 
\ell_{C}(\varphi (g))=\ell_{C}(\varphi (s_1)\ldots \varphi (s_k))\leq\sum_{i=1}^k\ell_{C}(\varphi(s_i))= k = \ell( g ). 
\]  
This shows $\varphi$ is a metric homomorphism. 
Since it is a $K$-$Q$-almost-homomorphism, 
we see that $(G,\ell )$ satisfies Definition \ref{mres}, i.e. is metrically fully  residually $\mathcal{C}$. 
\end{proof}

\begin{remark} 
{\em Note that in the situation 
of this subsection the corresponding version 
of Remark \ref{quasi} also holds. 
In Theorem \ref{net3} and Corollary \ref{r_f_2} we can replace the free group $\mathsf{F}_n$ by $\mathsf{F}^w_n$, where $w(z_1,z_2,\ldots, z_m )$ is a quasi-linear group word, $V_w$ is the variety of groups satisfying the corresponding  identity and $\mathsf{F}^w_n$ is the free $n$-generated group of $V_w$ ($n \not= \omega$). 
In Corollary \ref{r_f_2} we assume that $G$ is finitely presented in $V_w$. 
} 
\end{remark}

\subsection{LEF and the finite model property} 

If $(G,\ell )$ is an (invariant) pseudo-$\Lambda$-normed group, then it can be viewed as a first-order structure of the following form: 
\[
G^{\Box ,\mathbb{Q}} = (G, \cdot, \{ R^{\Box\varepsilon} (x) \, | \, \Box \in \{ <,>,= \} \, , \, \varepsilon \in \Lambda \cap \mathbb{Q} \} ).  
\] 
We interpret $R^{\Box\varepsilon} (x)$ as follows: 
\[ 
G^{\Box ,\mathbb{Q}} \models R^{\Box\varepsilon}(g) \, \Leftrightarrow (G, \ell ) \models \ell (g) \Box \varepsilon  \mbox{, where } \, \Box \in \{ <,>,= \}. 
\] 
Let $L^{\Box , \mathbb{Q}}$ be the following extension of the language of group theory: 
\[ 
( \cdot ,^{-1} , 1 , \{ R^{\Box \varepsilon} (x ) \, | \, \Box \in \{ <,>,=\} , \varepsilon \in \Lambda \cap \mathbb{Q} \} ). 
\]    
The following lemma is a standard fact from logic (see Chapter 4.1 in \cite{Keisler}).

\begin{lemma} \label{ultraprod}
Let $\mathcal{K}$ be a class of 
$L^{\Box , \mathbb{Q}}$-structures. 
Then an $L^{\Box , \mathbb{Q}}$-structure $M$ embeds into an ultraproduct of members of $\mathcal{K}$ if and only if $M\models Th_{\forall} (\mathcal{K})$. 
\end{lemma} 

From now on in this section let $\mathcal{C}$ be a class of invariant pseudo-$\Lambda$-normed groups. 
Let $\mathcal{C}^{\Box,\mathbb{Q}}$ be the class of all $L^{\Box , \mathbb{Q}}$-structures of the form $G^{\Box , \mathbb{Q}}$ for all $(G,\ell )\in \mathcal{C}$. 

\begin{theorem} \label{LE-ultra} 
Let $(G,\ell )$ be an invariant pseudo-$\Lambda$-normed group. \\ 
Then $G$ is metrically LE$\mathcal{C}$ if and only if $G^{\Box , \mathbb{Q}}$ is embeddable into an ultraproduct of elements of $\mathcal{C}^{\Box , \mathbb{Q}}$. 
\end{theorem} 

\begin{proof} 
Necessity. 
Let 
\[ 
I = \{ (K,Q) \, | \, K \subset_{fin} G \, , \, 
Q\subset_{fin} \Lambda \cap \mathbb{Q} \} 
\] 
and let $\mathcal{U}$ be an ultrafilter over $I$ containing all sets of the following form:
\[ 
J_{K\times Q}=\{(K',Q') \in I \, | \, K\subseteq K' , \, Q\subseteq Q'\}. 
\]
For each pair $(K,Q)\in I$ take a $K$-$Q$-almost-homomorphism 
$\xi_{K, Q}$ from $(G,\ell )$ to 
some $(G_{K, Q},\ell_{K,Q}) \in\mathcal{C}$ such that $\ell$ and $\ell_{K,Q}$ agree for $\square$-in/equalities over $K\times Q$. 
Now it is easy to see that $G^{\square,\mathbb{Q}}$ embeds into the ultraproduct $\prod_{\mathcal{U}} G^{\Box,\mathbb{Q}}_{K, Q}$ with respect to the map 
$\xi$: $\xi(x)=[\xi_{K, Q}(x)]_\mathcal{U}$.

Sufficiency. 
Assume that $G^{\square,\mathbb{Q}}$ embeds into an ultraproduct, say $M$, of structures from $\mathcal{C}^{\square,\mathbb{Q}}$. 
Having finite 
$K=\{ g_1 ,\ldots , g_n \} \subset G$ and $Q\subset\Lambda \cap \mathbb{Q}$ let $\Phi (g_1,g_2,\ldots,g_n)$ be the conjunction of all atomic formulas and/or their negations giving the multiplication table over $K$ and relations $R^{\square \varepsilon} (k)$ for all pairs $(g,\varepsilon)\in K\times Q$, which are satisfied in $G^{\Box,\mathbb{Q}}$. 
Thus $M\models \Phi (\bar{g})$ and there is an index, say $(K',Q') \in I$ and a structure $G^{\Box, \mathbb{Q}}_{K',Q'}$ which appears in $M$ such that it corresponds to a group from $\mathcal{C}$ 
and  
$G^{\Box, \mathbb{Q}}_{K',Q'}\models \Phi (\pi_{K',Q'}(\bar{g}))$, where $\pi_{K',Q'}$ is the corresponding projection. 
This projection is a $K$-$Q$-almost-homomorphism required by Definition \ref{mLEF}. 
\end{proof} 

It is obvious that under the assumptions of this section  a group $(G,\ell )$ is metrically LE$\mathcal{C}$ if and only if  
every finitely generated subgroup of $G$ is metrically LE$\mathcal{C}$. 
Applying this, Theorem \ref{LE-ultra} and Lemma \ref{ultraprod} we derive the following corollary. 

\begin{corollary} 
Let 
$(G,\ell )$ be a invariant pseudo-$\Lambda$-normed group. 
Then the following conditions are equivalent.  
\begin{enumerate}
\item $(G,\ell )$ is metrically LE$\mathcal{C}$.
\item $G^{\square,\mathbb{Q}}$ is embeddable into an ultraproduct of elements of $\mathcal{C}^{\Box , \mathbb{Q}}$. 
\item $G^{\square,\mathbb{Q}}\models Th_\forall(\mathcal{C}^{\square,\mathbb{Q}})$. 
\item Every finitely generated subgroup of $G$ is metrically LE$\mathcal{C}$.  
\end{enumerate} 
\end{corollary}

The case when $\mathcal{C}$ consists of finite structures is basic for us. 
Let us consider the finite model property for structures of $L^{\Box , \mathbb{Q}}$. 
We remind the reader that a structure $M$ has the finite model property if any sentence which holds in $M$ also holds in a finite structure. 
The following definition introduces some modifications of this property. 

\begin{definition} 
{\em 
Let $T$ be a theory of a language $L$ and $\Sigma$ be a family of sentences of this language. 
A structure $M$ has} 
the finite model property for $\Sigma$ with respect to $T$ 
{\em if any sentence from $\Sigma$ which holds in $M$ also holds in a finite model of $T$. }
\end{definition} 

It is easy to see that a a structure $M$ has the finite model property for existential sentences with respect to a thery $T$ if and only if it embeds into an ultraproduct of finite models of $T$. 
For example property LEF for abstract groups can be reformulated as follows. 
\begin{quote} 
Let $T_{UG}$ be the universal theory of groups. 
Then a group $G$ is LEF if and only if it has the finite model property for existential sentences with respect to $T_{UG}$. 
\end{quote} 
It is curious that this statement remains true for the language $L^{\Box,\mathbb{Q}}$ too. 
 
\begin{proposition} \label{fmpTG} 
Let $(G,\ell )$ be an invariant (pseudo) $\Lambda$-metric group. 
Then $G^{\Box , \mathbb{Q}}$ has the finite model property for existential sentences with respect to $T_{UG}$ if and only if the group $G$ is LEF as an abstract group. 
\end{proposition} 

\begin{proof} 
Necessity is obvious. 
To see sufficiency assume that $G$ is an infinite group which is LEF as an abstract group. 
Let 
$(\exists x_1, \dots ,x_n )\phi (\bar{x})$ 
be an existential 
$L^{\Box , \mathbb{Q}}$-sentence 
satisfied in $G^{\Box , \mathbb{Q}}$, where $\phi(\bar{x})$ is quantifier-free. 
W.l.o.g. we may assume that 
$\phi (\bar{x})$ is a conjunction of atomic formulas and/or their negations. 
Furthermore, we will assume that the purely group theoretic part of $\phi(\bar{x})$ describes the isomorphic type of some $n$-tuple from $G$.  

Let $(g_1, \ldots ,g_n )$ be a realization of $\phi (\bar{x})$ in 
$G^{\Box , \mathbb{Q}}$. 
As $G$ is LEF there is a finite group, say $G_0$, which contains $\bar{g}$ as a substructure (induced by $G$).  
Let us define in $G_0$ all relations of the form $R^{\Box \varepsilon} (g_i )$ exactly as they are defined in $G^{\Box , \mathbb{Q}}$. 
For the remaining elements of $G_0$ predicates $R^{\Box \varepsilon}(x)$
are defined arbitrarily. 
Then we see that $\bar{g}$ realizes $\phi (\bar{x})$ in the $L^{\Box , \mathbb{Q}}$-structure defined on $G_0$. 
\end{proof}

\begin{remark} 
{\em Connections of LEF with the finite model property are natural. 
Sections 2 and 3 of  \cite{PestovKw} contain formulations developing the one stated before Proposition \ref{fmpTG}.  
It is also shown in \cite{Ivanov} that LEF for a finitly generated group $G$ is equivalent to the finite model property for the Cayley graph of $G$. } 
\end{remark}  

We now want a counterpart of Proposition \ref{fmpTG} which would characterize metric LEF. 
We need an extension of $T_{UG}$ which will replace $T_{UG}$ in the corresponding formulation. 

We start with the following observation. 
Let $M$ be an $L^{\Box ,\mathbb{Q}}$-structure. 
Define a function 
$\mathsf{f}_M: M\times (\Lambda \cap \mathbb{Q}) \to \{ <,=,>\}$ as follows: 
\[ 
\mathsf{f}_M (a,\varepsilon ) = \Box \, \Leftrightarrow \, M \models R^{\Box \varepsilon} (a). 
\] 
In the case of structures $G^{\Box, \mathbb{Q}}$ this function is total. 
Note that the $L^{\Box ,\mathbb{Q}}$-sentence stating this is universal: 
\[
\forall x \bigvee \{ (R^{< \varepsilon}(x))^{\alpha} \wedge (R^{= \varepsilon}(x))^{\beta} \wedge (R^{> \varepsilon}(x))^{\gamma} \, | \, 
(\alpha, \beta, \gamma) \in \{ 0,1\}^3\} .
\]
Let $T_W$ be the theory of all $L^{\Box, \mathbb{Q}}$-structures with total $\mathsf{f}_M$ which are weight functions, 
see Definition \ref{weight}. 

\begin{lemma} 
$T_W$ is universally axiomatizable. 
\end{lemma} 

\begin{proof} 
It is easy to see that statements (1) - (4) of Definition \ref{weight} can be expressed by universal $L^{\Box ,\mathbb{Q}}$-sentences. 
\end{proof} 

\begin{itemize} 
\item Let $T_{PMG}$ be the extension of $T_{UG} \cup T_W$ by sentences saying that $\mathsf{f}_G$ belongs $PsL_{\bar{\Lambda}} (G)$. 
\item Let $T_{IPMG}$ be the extension of $T_{UG} \cup T_W$ by sentences saying that $\mathsf{f}_G$ belongs $IPsL_{\bar{\Lambda}} (G)$. 
\item Let $T_{IMG}$ be the extension of $T_{IPMG}$ by sentences saying that $\mathsf{f}_G$ belongs $IL_{\bar{\Lambda}} (G)$.  
\end{itemize}

\begin{lemma} 
Theories $T_{PMG}$, $T_{IPMG}$ and $T_{IMG}$ are universally axiomatizable. 
\end{lemma} 

\begin{proof} 
The sentence $R^{=0}(1)$ says that 
$\mathsf{f}_G (1,0)\in \{ = \}$.
The family of sentences of the form  
$\forall x (R^{\square q}(x)\leftrightarrow R^{\square q}(x^{-1}))$ says that for any $q\in \Lambda \cap \mathbb{Q}$ the property
$\forall x (\mathsf{f}_G(x,q) =\mathsf{f}_G (x^{-1},q))$ is satisfied.
Thus it is satisfied for any $r\in \Lambda$. 

To axiomatize the triangle condition for 
$\ell =\mathsf{w}_{\mathsf{f}_G}$ 
take all formulas of the following form:  
\[ 
\forall x,x' (\bigvee \{ R^{\Box q}(x) \, | \Box\in \{ <,=\} \} \wedge \bigvee \{ R^{\Box q'}(x') \, | \Box\in \{ <,=\} \}  \to 
\]
\[ 
\hspace{3cm} \bigvee \{ R^{\Box (q+q')}(xx') \, | \Box\in \{ <,=\} \} ) \mbox{ , where } q,q' \in \Lambda \cap \mathbb{Q} .
\]  
Invariantness is expressed as follows: 
$\forall x,y (R^{\square q}(x) \leftrightarrow R^{\square q}(yxy^{-1}))$. 

The class $IL_{\bar{\Lambda}}(G)$ is defined as follows:   
\[ 
\forall x (\bigvee \{ R^{\Box 0}(x) \, | \Box \in \{ <, = \}\}  \to x = 1 ). 
\] 
\end{proof} 

We now mention the following curious observation. 

\begin{proposition} 
An invariant pseudo-normed group $(G,\ell )$ is metrically LEF if and only if $G^{\Box , \mathbb{Q}}$ has the finite model property for existential sentences with respect to $T_{IPMG}$. 
If $\ell$ is a norm then in this statement $T_{IPMG}$ can be replaced by $T_{IMG}$.  
\end{proposition} 

\begin{proof} 
If $(G,\ell )$ is metrically LEF then  
by Theorem \ref{LE-ultra} 
the structure $G^{\Box , \mathbb{Q}}$ embeds into an ultraproduct of some structures 
$G^{\Box , \mathbb{Q}}_i$ defined for finite invariant pseudo-normed groups $(G_i ,\ell_i )$, $i \in I$. 
This obviously implies the finite model property as in the formulation. 

Assume that $G^{\Box , \mathbb{Q}}$ has the finite model property for existential sentences with respect to $T_{IPMG}$. 
Take a finite $K\subset G$ and a finite 
$Q\subset \mathbb{Q}$ with $0\in Q$. 
Using the finite model property find a finite group $G_{K,Q}$ and an  
$L^{\Box,\mathbb{Q}}$-structure on it, say $M$, which is a model of $T_{IPMG}$ such that there is a $K$-$Q$-almost-homomorphism 
$\xi_{K,Q}: (G,\ell ) \to (G_{K, Q},\mathsf{w}_{\mathsf{f}_M})$. 

Note that it can happen that for some 
$h\in G_{K,Q}$ the set 
$\{ q\in \Lambda \cap \mathbb{Q} \, | \, M\models R^{<q}(h) \vee R^{=q}(h) \}$ is empty, i.e. $\mathsf{w}_{\mathsf{f}_M}(h)$ can be undefined.   
In order to remedy this let us take any $q^M \in \mathbb{Q} \cap \Lambda$ with  
\[ 
q^M \ge \mathsf{max} (\{ \mathsf{w}_{\mathsf{f}_M} (x) \, | x \in G_{K,Q} , \, \mathsf{w}_{\mathsf{f}_M} (x) \in \Lambda \} \cup Q ),  
\] 
where we put $q^M$ to be equal to the right part only in the case when the value of this right part is $\mathsf{max}(\Lambda)$. 
Let us define a pseudo-norm, say $\ell_{K,Q}$, on $G_{K,Q}$ so that $\xi_{K,Q}$ is a $K$-$Q$-almost-homomorphism    
into $(G_{K,Q},\ell_{K,Q})$. 
When $h \in G_{K,Q}$ and $\mathsf{w}_{\mathsf{f}_M}(h)$ is defined we put $\ell_{K,Q}(h) = \mathsf{w}_{\mathsf{f}_M}(h)$.
When $\mathsf{w}_{\mathsf{f}_M}(h)$ is undefined we put $\ell_{K,Q}(h) = q^M$.   

Since $M\models T_{IPMG}$, the function $\mathsf{w}_{\mathsf{f}_M}$ satisfies (i) and (ii) of Definition \ref{lf}. 
Thus so does $\ell_{K,Q}$. 
Similarly we see that $\ell_{K,Q}$ is invariant. 
In order to verify the triangle condition assume that  
\[ 
\ell_{K,Q} (g) \le q  \, \mbox{ and }  \ell_{K,Q} (h)  \le q'  \mbox{, where } 
\mathsf{max}(q,q') < q^M.    
\]
Then $\mathsf{w}_{\mathsf{f}_M}(gh)$ is defined by $T_{IPMG}$ and we see $\ell_{K,Q}(gh) \le q + q'$ . 
If 
\[ 
\ell_{K,Q} (g) \le q  \, \mbox{ and }  \ell_{K,Q} (h)  \le q' \mbox{, where } 
q= q^M \mbox{ or } \, q' = q^M,     
\]
Then by the definition of $\ell_{K,Q}$ we have $\ell_{K,Q}(gh) \le q^M$, i.e 
$\ell_{K,Q}(gh) \le q+ q'$. 

The case when $G^{\Box , \mathbb{Q}}$ has the finite model property for existential sentences with respect to $T_{IMG}$ is similar. 
\end{proof}

\subsection{LE$\mathcal{C}$ and direct limits} 

The following theorem is a counterpart of Theorem 2.5 of \cite{OHP}. 
It roughly says that verifying metric LE$\mathcal{C}$ for a normed group $(G,\ell )$ we can replace it by some $(L_n, \ell_n )$ 
where $L_n$ is residually $\mathcal{C}$ and $\ell_n$ is a ``fragment" of $\ell$. 

\begin{theorem} 
Let $\mathcal{C}$ be a class of (invariant) pseudo-$\Lambda$-normed groups and $(G,\ell )$ be an (invariant) pseudo-$\Lambda$-normed group. 
Then the following conditions are equivalent: 
\begin{itemize} 
\item $(G,\ell )$ is metrically LE$\mathcal{C}$;  
\item for every finitely generated subgroup $L\le G$ there is a sequence of pseudo-normed groups $(L_n , \ell_n)$, $n\in \mathbb{N}$, together with isometric homomorphisms 
$\varphi_n: (L_n,\ell_n)\rightarrow (L_{n+1},\ell_{n+1})$, $n\in \mathbb{N}$, 
such that 

(i) treating 
$\{ L_n \, | \, n\in \mathbb{N}\}$ as a direct system under the system of homomorphisms of the form 
$\varphi_{n,m} = \varphi_m \circ \ldots \circ \varphi_n$, $n<m \in \mathbb{N}$, 
the group $L$ is the direct limit of $\{ L_n \, | \, n\in \mathbb{N}\}$  
(for the next property let us fix the corresponding isometric homomorphisms $\psi_n :L_n \rightarrow L$, $n\in \mathbb{N}$); 

(ii) for every $n\in \mathbb{N}$ and any $E\subset_{fin} L$ and $Q\subset_{fin} \Lambda \cap \mathbb{Q}$ such that $0\in Q$ and the homomorphism $\psi_n :L_n \rightarrow L$ is injective on some $E'\subset L_n$ with $\psi_n (E') = E$, there is a homomorphism from $L_n$ to some $(C,\ell_C)\in \mathcal{C}$ which is a $E'$-$Q$-almost-homomorphism;  

(iii) the groups $L_n$ are residually $\mathcal{C}$ as abstract groups. 
\end{itemize} 
\end{theorem} 

Before the proof we remind the reader that (i) means that $L$ can be viewed as $(\sqcup_n L_n)/\mathtt{\sim}$, where for $x\in L_n$ i $y\in L_m$ the condition $x \mathtt{\sim} y$ states that there is $k\geq\max\{n,m\}$ such that $\varphi_{n,k}(x)=\varphi_{m,k}(y)$. 
Then each $x\in L$ is represented as a sequence 
$\{ x_n \, | \, x_n \in L_n , \, n\in \mathbb{N}\}$, where all $x_n$ are pairwise  $\mathtt{\sim}$-equivalent. 

\begin{proof} 
This proof is a modification of the proof of Theorem 2.5 from \cite{OHP}. 
Necessity. 
Let $L$ be a finitely generated subgroup of $(G,\ell )$ and 
\[ 
L=\langle a_1,\ldots,a_p \, | \, r_0 (\bar{a}),\ldots,r_i (\bar{a}),\ldots \rangle . 
\] 
For each $i\in \mathbb{N}$ let 
\[ 
D_i =\langle a_1,\ldots,a_p \, | \, r_0 (\bar{a}),\ldots,r_i (\bar{a})\rangle . 
\] 
Assuming that $x_1,\ldots,x_p$ are free generators of the free group 
$\langle x_1,\ldots, x_p\rangle$,  
\[ 
D_i\cong \langle x_1,\ldots,x_p\rangle /N_i , 
\]  
where $N_i$ is the normal subgroup generated by  $r_0(\bar{x}),\ldots,r_i(\bar{x})$ (with $\bar{x}=(x_1,\ldots,x_p)$). 
For $i \le j\in \mathbb{N}$ let  
$f_{i,j}$ be a homomorphism $D_i \to D_j$, defined by $f_{i,j}(w N_i )=w N_j$. 
Then $L=\lim\limits_\to D_i$.

Let 
\[ 
K_n = \bigcap \{ M \triangleleft D_n \, | \, D_n/M \mbox{  embeds into some } (C,\ell_C )\in\mathcal{C} \} ,   
\]  
$L_n = D_n/K_n$ and let 
$\pi_n:D_n\rightarrow L_n$ be the natural homomorphism. 
It is easy to see that $L_n$ is residually $\mathcal{C}$ as an abstract group.
 
Note that when 
$w \in D_{i+1} \setminus M$ with $M\triangleleft D_{i+1}$ such that $D_{i+1}/M$ embeds into some $(C,\ell_C )\in \mathcal{C}$ then 
$f^{-1}_{i,i+1} (w) \cap f^{-1}_{i,i+1} (M)=\emptyset$ and 
$D_i /f^{-1}_{i,i+1} (M) \cong D_{i+1}/M$. 
In particular, when $w\in D_{i+1}\setminus
K_{i+1}$ we have $f^{-1}_{i,i+1} (w) \cap K_i = \emptyset$. 
Thus for each pair $i \le j$ there is a natural  homomorphism 
$\varphi_{i,j}: L_i\rightarrow L_j$ such that 
\[ 
\pi_j\circ f_{i,j}=\varphi_{i,j}\circ\pi_i.
\] 
Let 
$\mathcal{C}L=\lim\limits_\to L_n$. 

Consider the natural homomorphism 
$\pi: L\rightarrow\mathcal{C}L$ induced by the family $\pi_i$, $i\in \mathbb{N}$. 
See the following diagram:
\[
\xymatrix@R=2.9pc @C=2.9pc
{
{D_0}\ar@{->}[r]^{f_{0,1}}\ar@{->}[d]^{\pi_0} & {D_1}\ar@{->}[r]\ar@{->}[d]^{\pi_1} & {\ldots}\ar@{->}[r] & {D_n}\ar@{->}[r]^{f_{n,n+1}}\ar@{->}[d]^{\pi_n} & {D_{n+1}}\ar@{->}[r]\ar@{->}[d]^{\pi_{n+1}} & {\ldots} & {L}\ar@{->}[d]^{\pi} \\
{L_0}\ar@{->}[r]^{\varphi_{0,1}} & {L_1}\ar@{->}[r] & {\ldots}\ar@{->}[r] & {L_n}\ar@{->}[r]^{\varphi_{n,n+1}} & {L_{n+1}}\ar@{->}[r] & {\ldots} & {\mathcal{C}L}
}
\]  
It is clear that $\pi$ is a surjection.
Furthermore, for any 
$w(\bar{x})\in \langle x_1 ,\ldots ,x_p  \rangle$,  
\[ 
\pi(w(\bar{a}))=1 \Leftrightarrow 
\exists n\in\mathbb{N} (\pi_n(w(\bar{a}))=1).
\] 
To see that $\pi$ is an isomorphism take any $g\in L\backslash\{1\}$ and an appropriate word $w(\bar{x})$ with  
$g=w(\bar{a})$. 
Then for every $m\in\mathbb{N}$,  
\[ 
L\models 
\exists\bar{x}\big(w(\bar{x})\neq 1\wedge\bigwedge \{ (r_i(\bar{x})=1) \, | \, i\in\{0,1,\ldots,m\} \} \big) .
\] 
Since $(L,\ell )$ is LE$\mathcal{C}$ 
there exists $(C_m ,\ell_m )\in\mathcal{C}$ such that this sentence holds in   
$C_m$.  
Let $\bar{b}_m\in C_m$ be a realization of the quantifier-free part. 
Then the map $\bar{a} \rightarrow \bar{b}$ defines a homomorphism $D_m \rightarrow C_m$. 
Thus there is $M_m \triangleleft D_m$ such that 
$D_m/M_m\cong\langle \bar{b}_m \rangle \leq C_m$. 
Since $K_m\leq M_m$ we see that   
$\pi_m(w(\bar{a}))\neq1$. 
Applying this argument for all $m\in\mathbb{N}$ we have $\pi(g)\neq 1$. 
Thus $\pi$ is an isomorphism.

Let $\hat{\psi}_i: D_i\to L$, $i\in \mathbb{N}$, be the natural family of maps realizing $L=\lim\limits_\to D_i$ and let  
$\psi_i:L_i\to L$, $i\in \mathbb{N}$, be the corresponding family induced by an application of $\pi^{-1}$.
Let us define pseudo-norms on $D_i$ and $L_i$ as follows:
\[
\ell_{D_i}(x)=\ell(\hat{\psi}_i (x)) 
\mbox{ and }
\ell_{L_i}(\pi_i (x))=\ell(\psi_i( \pi_i (x))) \mbox{ , where }x\in D_i . 
\] 
We in particular see that homomorphisms  $\hat{\psi}_i$ and $\psi_i$ are isometric.  
Thus so are homomorphisms $\pi_i$, $i\in \mathbb{N}$. 
We see that the statement holds with the  possible exception of condition (ii). 
 
In order to verify (ii) take 
$E'=\{g_1,\ldots, g_k\}\subset L_n$ and a finite $Q\subset\mathbb{Q}\cap \Lambda$ with $0\in Q$. 
We denote $E=\psi_n (E')$. 
Assume that $\psi_n$ is injective on $E'$, i.e. $|E|=|E'|$.

Let $\psi_n (g_i)=w_i (\bar{a})$, $1\leq i\leq k$. 
For each $q\in Q$ let 
$\square_{i,q}\in \{ <,=, > \}$ be chosen so that $(L,\ell)$ satisfies the following sentence:
\[ 
\big(\bigwedge_{i\leq n}(r_i(\bar{a})=1)\wedge \bigwedge_{i<j\le k} (w_i(\bar{a})\neq w_j(\bar{a}))
\wedge
\bigwedge_{\substack{ q\in Q \\ i\leq k }} (\ell (w_i(\bar{a}))\square_{i,q} q)\big).  
\] 
Then there is $(C_n,\ell_n)\in\mathcal{C}$ which satisfies the corresponding sentence 
\[ 
\exists\bar{x}\big(\bigwedge_{i\leq n}(r_i(\bar{x})=1)\wedge \bigwedge_{i<j\le k} (w_i(\bar{x})\neq w_j(\bar{x}))
\wedge
\bigwedge_{\substack{q\in Q \\ i\leq k }} (\ell_n (w_i(\bar{x}))\square_{i,q} q)\big).  
\] 
Let $\bar{b} \in C_n$ be a realization in $(C_n,\ell_n )$ of the quantifier-free part of it. 
Arguments as above show that the map $\bar{a}\to\bar{b}$ induces a  homomorphism $L_n\to C_n$ satisfying (ii) of the formulation.

Sufficiency. 
In order to prove LE$\mathcal{C}$ take finite $E\subset G$ and $Q\subset\mathbb{Q}\cap \Lambda$ with $0\in Q$. 
Let $L=\langle E \rangle$.   
Then $L$ is a finitely generated pseudo-normed subgroup of $(G,\ell )$ with respect to the corresponding restriction of $\ell$. 
Let $\{(L_n,\ell_n)\}_{n\in\mathbb{N}}$ be a system 
as in the formulation, i.e. in particular   
$L=\lim\limits_{\rightarrow}  L_n$. 

Find $n\in\mathbb{N}$ such that there is  $E'\subseteq L_n$ with $\psi_n$ to be an injective isometry from $E'$ to $E$. 
As a result $\psi_n$ is a $E'$-$Q$-almost-homomorphism.
Using the sufficiency condition find a homomorphism $\varphi$ from $(L_n,\ell_n)$ to some $(C,\ell_C )\in\mathcal{C}$ which is a $E'$-$Q$-almost-homomorphism injective on $E'$.

Define $\psi: G\to C$ as follows:
\[ 
\psi(g)=\left\{
\begin{array}{cl}
\varphi(g')&\mbox{if there exists } g'\in E': \psi_n(g')=g\\
1&\mbox{otherwise.}
\end{array}
\right.
\] 
It is easy to see that this is a $E$-$Q$-almost-homomorphism. 
Since $E$ and $Q$ were arbitrary we se that $(G,d)$ is LE$\mathcal{C}$.
\end{proof}

\bigskip 

{\bf Acknowledgment.}
The authors are grateful to Oleg Bogopolski for the reference \cite{thomover}.

Aleksander Iwanow

Department of Applied Mathematics, Silesian University of Technology, 

ul. Kaszubska 23, 44-101 Gliwice, Poland 

Aleksander.Iwanow@polsl.pl

\bigskip 
Ireneusz Sobstyl

Faculty of Pure and Applied Mathematics, 

Wroc{\l}aw University of Science
and Technology, 

Wybrzeze Wyspianskiego Str. 27, 50-370, 
Wroc{\l}aw, Poland

ireneusz.sobstyl@pwr.edu.pl

\end{document}